\documentclass[a4paper,10pt]{scrartcl}
\usepackage[utf8]{inputenc}
\usepackage[english]{babel}
\usepackage[fixlanguage]{babelbib}
\usepackage{cite}
\usepackage{amssymb, amsmath, amstext, amsopn, amsthm, amscd, amsxtra, amsfonts}
\usepackage{yfonts, mathrsfs}
\usepackage{enumerate}
\usepackage{graphicx}%
\usepackage{colonequals}
\usepackage{tikz}
\usepackage{tabularx}
\usepackage{booktabs}
\usepackage{todonotes}
\usepackage{adjustbox}
\usepackage{placeins}
\usetikzlibrary{cd}
\usepackage{hyperref}%
\hypersetup{
urlcolor = red,
colorlinks = true,
linkcolor = blue,
citecolor = blue,
linktocpage = true,
pdftitle = {Topological splitting of the space of meromorphic germs},
pdfauthor = {Rafael Dahmen, Sylvie Paycha, Alexander Schmeding},
bookmarksopen = true,
bookmarksopenlevel = 1,
unicode = true,
hypertexnames =false
}%
\usepackage{cleveref}
\swapnumbers
\newtheorem{thm}{Theorem}[section]
\newtheorem{la}[thm]{Lemma}
\newtheorem{prop}[thm]{Proposition}

\newtheorem{cor}[thm]{Corollary}

\theoremstyle{definition}
\newtheorem{defn}[thm]{Definition}
\newtheorem{exa}[thm]{Example}
\newtheorem{coexa}[thm]{Counter example}
\newtheorem{rem}[thm]{Remark}

\newcommand{\calm}{\mathcal M}

\newcommand{\N}{{\mathbb N}}
\newcommand{\R}{{\mathbb R}}
\newcommand{\Z}{{\mathbb Z}}

\newcommand{\C}{{\mathbb C}}

\newcommand{\newLa}{\mathcal{L}}
\DeclareMathOperator{\BnewLa}{\mathbf{\newLa}}

\DeclareMathOperator{\id}{id}

\newcommand{\coloneq}{\colonequals}

\DeclareMathOperator{\Supp}{Supp}

\DeclareMathOperator{\pr}{pr}
\DeclareMathOperator{\ev}{ev}
\DeclareMathOperator{\reg}{reg}

\DeclareMathOperator{\Indep}{Indep}
\DeclareMathOperator{\Dep}{Dep}

\newcommand{\Frechet}{Fr\'{e}chet}

\date{}

\newcommand{\Poly}{\mathcal P}

\newcommand{\Hol}{\operatorname{Hol}}
\newcommand{\BHol}{\operatorname{BHol}}
\newcommand{\Holo}{\mathcal{H}}
\newcommand{\CHolok}{\Holo^{\C}(\C^k)}
\newcommand{\RHolok}{\Holo(\C^k)}
\newcommand{\CHoloN}{\Holo^{\C}(\C^\N)}
\newcommand{\RHoloN}{\Holo(\C^\N)}

\newcommand{\Mer}{\operatorname{Mer}}
\newcommand{\Mero}{\mathcal{M}}
\newcommand{\CMerok}{\Mero_{\newLa}^{\C}(\C^k)}
\newcommand{\RMerok}{\Mero_{\newLa}(\C^k)}
\newcommand{\polarMerok}{\Mero_{\newLa,Q}^-(\C^k)}

\newcommand{\RMeroN}{\Mero_{\BnewLa}(\C^\N)}
\newcommand{\polarMeroN}{\Mero_{\BnewLa,Q}^-(\C^\N)}
\begin{document}
  \title{A  topological splitting of  the space of meromorphic germs in several variables  and continuous  evaluators}  
  \author{Rafael Dahmen, Sylvie Paycha and Alexander Schmeding}
  \maketitle
  \abstract{ We prove a topological decomposition of the space of meromorphic germs at zero in several variables with prescribed linear poles as a sum of  spaces of holomorphic and polar germs. Evaluating the resulting  holomorphic projection at zero gives rise to  a continuous evaluator (at zero) on the space of meromorphic germs  in several variables. Our constructions are  carried out in the framework of Silva spaces and use an inner product on the underlying space of variables. They generalise to several variables, the topological direct decomposition of  meromorphic germs at zero  as sums of holomorphic and polar germs previously derived by the first and third author and provide a topological refinement of a known algebraic decomposition of such spaces previously derived by the second author and collaborators.}

  \\[.25em]
  
  
  \tableofcontents

 \section*{Introduction}\addcontentsline{toc}{section}{Introduction}
  Meromorphic functions in several variables with linear poles, so functions on $\C^k$ of the type 
  \[
     f(z_1, \cdots, z_k)=\frac{h(z_1, \cdots, z_k)}{L_1(z_1, \cdots, z_k)\,\cdots \, L_m(z_1, \cdots, z_k) }, \quad k\in \N, m\in \Z_{\geq 0},
  \]
  where $h$ is a holomorphic function  and $L_1, \cdots, L_m$ are linear forms with real coefficients on $\C^k$, are ubiquous in mathematics and physics. 
  They arise in quantum field theory from Feynman integrals \cite{Speer,DZ,CGPZ2}, in number theory from multizeta functions, see e.g. \cite{MP}, in equivariant geometry from discrete Laplace transforms on polytopes  \cite{B,BV}, which are in turn related to discrete Laplace transforms on convex polyhedral cones \cite{GPZ2}.

  Depending on the context, the meromorphic functions have a specific type of linear poles e.g. $L_j(z_1, \cdots, z_k)
  =z_1+\cdots +z_j, 0\leq j\leq k$ in the context of multiple zeta functions, $L_J(z_1, \cdots, z_k)
  =\sum_{i\in J}z_i$ for subsets $J\subseteq \{1, \cdots, k\}$ in the context of Feynman integrals. In all the situations mentioned above, one can decompose the algebra $\RMerok $
  of meromorphic germs at zero on $\C^k$  (with real coefficients) and with a prescribed type of $\R$-linear poles given by a generating set $\newLa$. For this one chooses a suitable inner product $Q$ on $\C^k$ and obtains
  \begin{equation}\label{eq:splitting}
    \RMerok = \RHolok  \oplus \polarMerok 
  \end{equation} 
  as a direct sum of the algebra $\RHolok$ of holomorphic germs at zero, and a space $\polarMerok$ of polar germs at zero.    Such a decomposition was derived in \cite{BV} and \cite{GPZ} by means of an euclidean structure inner product $Q_k$ on the underlying spaces $\R^k$ with $k\in \N$. 
  This leads us to our first theorem (which results from Theorem \ref  {thm:MLambdaCN} and Proposition \ref{prop:piQ}):
  
  \smallskip\textbf{Theorem A.}  \emph{For every $k\in\N$, every index set $\newLa$ and suitable inner product $Q$ on $\C^k$ the space of meromorphic germs $\RMerok$ carries a natural topology such that it splits
  \[
     \RMerok = \RHolok  \oplus \polarMerok
  \]
  as a locally convex topological vector space. Here, the space $\RHolok$ is endowed with its natural topology.  In other words, the resulting projection $\pi_Q^k: \RMerok\to\RHolok$ onto $\RHolok$ parallel to $\polarMerok$ is continuous for any $k$ in $\N$. 
  }\smallskip 
  
  The decomposition is carried out in the context of Silva spaces, which we recall in the next section.
  By a limiting argument, the same holds for the germs on the infinite dimensional space $\C^\N$.
    Depending on the chosen inner product $Q$ on $\C^\N$, one can define a natural notion of orthogonality $\top_Q$  on the space of meromorphic germs.
  
  \smallskip\textbf{Theorem B.}
  \emph{
  Denoting by $\ev_0$ evaluation of holomorphic germs in $0$ and by $\pi_Q$ the projection induced by the inner product $Q$, one defines a map
  \[
     {\mathcal E}_Q^{\mathrm{MS}}\coloneq \ev_0\circ \pi_Q:\RMeroN\to \C.
  \]
  Then ${\mathcal E}_Q^{\mathrm{MS}}$ is continuous linear and partially multiplicative in the following sense
  \[
     f_1 \,\top_Q\, f_2\Longrightarrow {\mathcal E}_Q^{\mathrm{MS}}(f_1\cdot f_2)= {\mathcal E}_Q^{\mathrm{MS}}(f_1)\cdot {\mathcal E}_Q^{\mathrm{MS}}(f_2).
  \]
  } 
  
  We check that the resulting convergence for meromorphic germs arising from Feynman integrals indeed coincides with the one required by Speer in his pioneering work on analytic renormalisation  \cite{Speer}. In fact, the map ${\mathcal E}_Q^{\mathrm{MS}}$ in Theorem B is a continuous evaluator in the sense of Speer.

  \subsection*{A natural topological framework: Silva spaces}
  In order to endow the space of meromorphic germs at zero with  prescribed linear poles with  a topology, we need to make more precise the notion of  algebra of meromorphic germs at zero with prescribed linear poles. Given a positive integer $k$, from a countable set $\newLa$ of linear mappings $\C^k \rightarrow \C$ which does not contain the zero-map, we build  the semi-group (without unit)
  \[
     \mathcal{S}_{\newLa} \coloneq \{L_1 \cdot L_2 \cdots L_m \mid L_1, \ldots , L_m \in {\newLa} , m \in \N \}
  \]
  generated by the set $\newLa$. 
  For an element $P \colon \C^k \rightarrow \C$ of $\mathcal{S}_{\newLa}$, and $n\in \N$ we define the space
  \[
     \Mer_{P}^{\C} (B_{1/n}(\C^k)) \coloneq \left\{\frac{f}{P} \middle| f \in \BHol^{\C}(B_{1/n}(\C^k))\right\},
  \]
  where $B_{1/n}(\C^k) \coloneq \{z \in \C^k \mid \lVert z \rVert < 1/n\}$ is the (open) ball of radius $1/n$ in $\C^k$ and $\BHol^\C (U)$ is the space of bounded holomorphic functions on the set $U$. We obtain an inductive system $\left(\Mer_{P}^{\C}(B_{1/n}(\C^k))\right)_{P\in \mathcal S({\newLa}), n\in \N}$ of Banach spaces and  
  Proposition \ref{prop:lcxv-alg} states that
  the space of meromorphic functions on $\C^k$ with linear poles in $\newLa$ defined as the locally convex inductive limit 
  \[
     \CMerok \coloneq \lim_{\longrightarrow} \Mer_{P}^{\C} (B_{1/n}(\C^k))
  \]
  is a complex Silva space. We refer the reader to \cite{MR97702} for the concept of Silva spaces, also called DFS-space in the literature, which form a class of well-behaved locally convex spaces. One particularly important property of these spaces is that continuity of a mapping on a Silva space can be checked on the steps of the inductive limit (whence continuity can often be reduced to a problem involving only the more familiar setting of Banach spaces). Other remarkable properties which we shall use in the sequel are that countable limits and closed linear subspaces of Silva spaces are again Silva spaces, \cite{MR97702}.

  In the following, we shall restrict to meromorphic functions in $ \CMerok $ which take real values on  $\R^k$, therefore dropping the superscript $\C$ in the notations, writing $\RMerok$ and  $ \RHolok$.

  \subsection*{A topological splitting: main results}
  
  Let us now describe the results of the paper, leading to a topological version of  (\ref{eq:splitting}) for fixed $k$.
  A first step towards a topological splitting is Proposition  \ref{prop:lcxv-alg}, which shows that $\RMerok$ and  $ \RHolok$ are Silva spaces. 
  
  Endowing $\C^k$ with an inner product $Q$  such  that $Q|_{\R^k \times \R^k}$ takes values in $\R$, we consider the space $\polarMerok$  of polar germs (Definition \ref{defn:suppcones}) with a prescribed type of poles determined by the generating set $\newLa$.
  
  Our first main result is Theorem  \ref{polargerms_closed}, which   shows that the space $\polarMerok$ is a closed linear subspace of $\RMerok$. 
  This leads to  Theorem \ref{thm:splittingspaces} where we prove a topological splitting of $\RMerok$ as the topological direct sum of two locally convex spaces, the space $ \RHolok$ of  holomorphic functions and the space
  $\polarMerok$ generated by polar germs:
  \begin{equation}\label{eq:splittingMer}
    \RMerok = \sum_{P\in \mathcal{S}_{\newLa}}\RHolok\cdot P^{-1} = \RHolok \oplus \polarMerok.
  \end{equation} 
  It is surprisingly subtle to obtain the splitting of the inductive limit. While there are splitting operators on the level of the steps of the limit (cf.\ Lemma \ref{lem:splittingoff}), these are not amenable to an inductive argument.  
  
  We then consider the inductive limit of the above spaces and the corresponding splittings as $k\to \infty$. The projection $\mathrm{pr}_k\colon \C^{k+1} \rightarrow \C^k$ onto the first $k$ components yields a projective system of bonding maps which describe the locally convex direct product $\C^\N$. For each $k$ in $\N$, we assume that $L \circ \mathrm{pr}_k $ lies in $\newLa_{k+1}$ if $L$ is in $\newLa_k$ and define $\BnewLa = \bigcup_{k \in \N}\newLa_k$. This gives rise to an inductive system of locally convex spaces which allows us to construct an (locally convex) inductive limit 
  \begin{align*}
    \RMeroN &= \{f \colon \C^\N \rightarrow \C \mid \exists k \in \N, f_k \in \RMerok \text{ such that } f = f_k \circ \pi_k^\N \}\\
    &=\lim_{\longrightarrow} \RMerok .
  \end{align*} 
  Thanks to the stability of the class of Silva spaces under inductive limits, Propositions \ref{prop:HSilva} and \ref{prop:MLambdainfty} yield that the inductive limits $\RHoloN$ and
  $\RMeroN$ are both Silva spaces. 
  
  Putting all the ingredients together, leads to our main result in Theorem \ref{thm:MLambdaCN}, namely a refinement of the  algebraic splitting of $\RMeroN$ 
  derived in \cite{BV,GPZ}, resulting from an inductive limit of the splittings (\ref{eq:splittingMer}), to a topological splitting of locally convex spaces. Explicitly, by means of a  family $Q = (Q_k)_{k\in \N}$ of inner products $Q_k$ on $\C^k$, we decompose $\RMeroN$, which by Proposition \ref{prop:MLambdainfty} is a Silva space, as the locally convex sum 
  \[
     \RMeroN = \RHoloN \oplus \polarMeroN, \  \text{where } \RHoloN = \underset{\longrightarrow }{\lim}\RHolok
  \]
  is the Silva space of germs of holomorphic functions.

  \subsection*{A topological minimal subtraction scheme}
  
  As a consequence of the topological splitting, the canonical projection  $\pi_Q\colon \RMeroN\longrightarrow\RHoloN$ induced by the splitting 
  is a continuous linear map with a partial multiplicative property (\ref{eq:partialmult}).
  
  The canonical projection combined with the evaluation at zero $\ev_0$  to the continuous linear form (see Theorem \ref{thm:ev})
  \[
     \mathcal{E}^{\text{MS}}_Q:= \ev_0\circ \pi_Q,
  \]
  acting on meromorphic germs at zero.
  
  In one variable, the projection $\pi_Q$ amounts to deleting the principal part from the Laurent series:
  \[
     \pi_+\left(\sum_{k=-K}^\infty a_kz^k\right) = \sum_{k=0}^\infty a_kz^k,\quad K\in \Z_{\geq 0}
  \]
  which  gives rise to the map  $\mathrm{ev}_0\circ \pi_+$ underlying the minimal subtraction  scheme used (in various disguises) in quantum field theory, a method to extract finite parts from a priori divergent expressions,  which goes back to   \cite{THooft73} and \cite{Wein73}. The map  $\mathcal{E}_Q^{\text{MS}}=\mathrm{ev}_0\circ \pi^Q$, which we somewhat abusively refer to   as a \textbf{  minimal subtraction scheme},  is a particular instance of  a generalised $Q$-evaluator (see Definition \ref{defn:evaluator}), a definition inspired by Speer's generalised evaluators \cite{Speer} and used previously in work (in collaboration) by one of the authors \cite{GPZ11}. 
  
  In the present article we show that the minimal subtraction scheme is  topological in the sense that $\mathcal{E}^{\text{MS}}_Q$ is continuous. This gives a precise interpretation in the framework of $Q$-evaluators, of the continuity assumption Speer requires of his generalised evaluators.  {  The partial multiplicativity required in \eqref{eq:multFQ} of these operators can be interpreted in the  locality setup developed in \cite{CGPZ},  as a locality character property, with the locality relation given by the graph by the binary relation $\perp_Q$. Hence our results call for a generalisation of the algebraic locality setup of  \cite{CGPZ} to a topological locality setup, by which the locality is required to be compatible with the ambient topology. This raises interesting and challenging questions such as how to enhance locality tensor products on vector spaces (discussed in \cite{CFLP}) to locality tensor products on topological vector spaces. } 
  
  An open question is  the classification of continuous $Q$-evaluators on certain classes meromorphic germs.
  In work in progress  by the second author and collaborators, it is shown that  modulo the action of what the authors call the $Q$-Galois group, generalised $Q$-evaluators on certain classes of meromorphic germs are given by the $Q$-minimal subtraction scheme $\mathcal{E}_Q^{\text{MS}}$, thus showing the importance of the latter.\medskip

  
  \begin{figure}[h]
    \caption{\textbf{Notations for\ldots}}
    \begin{tabularx}{\textwidth}{@{}XX@{}}
      \toprule
      \textbf{Spaces of functions} \\
      $\Hol^\C (U)$ & holomorphic functions on an open set $U$ with values in $\C$\\
      $\Hol (U)$ & The real subspace of all elements in $\Hol^\C (U)$ mapping points in $\R^k$ to $\R$ \\
      $\BHol^\C (U)$ & Bounded holomorphic functions on an open set $U$ with values in $\C$\\
      $\BHol (U)$ & The real subspace of all elements in $\BHol^\C (U)$ mapping points in $\R^k$ to $\R$ \\
      $\Mer_P^\C (U)$ & Meromorphic functions on an open set $U$ with values in $\C$ and prescribed polynomial pole $P$\\
      $\CMerok $ & $\C$-valued germs of meromorphic functions in $0 \in \C^k$ with poles generated by $\newLa$ \\
      $\RMerok$ & The real subspace of all elements in $\CMerok$ mapping points in $\R^k$ to $\R$\\
      $\CHolok $ & $\C$-valued germs of holomorphic functions in $0\in \C^k$\\  
      $\RHolok \subseteq  \RMerok$ & $\C$-valued germs of holomorphic functions in $0\in \C^k$ mapping points in $\R^k$ to $\R$\\  
      $\polarMerok$ & Subspace of $\RMerok$ spanned by the polar germs\\
      &\\
      \textbf{Parameters}\\
      $\newLa$  & generating set of linear poles\\
      $\mathcal{S}_\newLa$ & semi-group generated by $\newLa$\\
      $Q_k$ & Inner product on $\C^k$\\ 
      &\\
      \textbf{Sets}
      \\
      $[[1,m]]$  &  The set of integers from $1$ to $m$ \\
      $\Dep(f)$ (resp.  Indep$(f)$) &  Dependence (resp. Independence) subspace of a germ/function\\ 
      $B_{r}(\C^k)$ & The open ball in $\C^k$ of radius $r$ centered at $0$ for the norm $\Vert\cdot \Vert = \sqrt{Q_k (\cdot,\cdot)}$ \\
      \bottomrule
    \end{tabularx}
  \end{figure}
  \FloatBarrier
  \newpage
 \section{Meromorphic germs on $\C^k$ as a Silva space}\label{sect:meroCk}
  In this section we construct the locally convex algebra of meromorphic germs with linear poles \cite{GPZ}. Let us first fix some notations and conventions.\smallskip
  
  \textbf{Notations}
  We set $\N = \{1,2,3,\ldots \}$ and let $k \in \N$. We choose and fix a sesquilinear inner product  $Q \colon \C^k \times \C^k \rightarrow \C$, linear in the first component. For later purposes we shall also assume that $Q|_{\R^k \times \R^k}$ takes values in $\R$.
  We will then write $\lVert \cdot \rVert = \sqrt{Q(\cdot,\cdot)}$ for the associated norm and $B_{r}(\C^k) \coloneq \{z \in \C^k \mid \lVert z \rVert < r\}$ the ball of radius $r>0$ in $\C^k$. Note that with these conventions we have $B_r(\C^k) \cap \R^k = B_r (\R^k)$. For a natural number $m\in\N$ we use the notation $[[1,m]]\coloneq\{1,\ldots,m\}$.
  
  The inner product $Q$  induces an inner product on the dual space $(\C^k)^* $ defined by  
  \begin{equation}\label{duality:scal}
    Q^*(L_1, L_2)\coloneq Q(L_1^*, L_2^*)\quad \forall L_1, L_2\in(\C^k)^*.
  \end{equation}
  The idea will be to construct the space of meromorphic functions as a certain limit of Banach spaces known as a Silva space.
  As these spaces are central to all that is following, we recall now basic facts on Silva spaces.
  
  \subsection{Preliminaries on Silva spaces}\label{rem:Silva}
  
  Let $(E_k,\iota_{k,k+1})_{k\in \N}$ be an inductive system of Banach spaces with compact linear bonding mappings\footnote{A continuous linear mapping between Banach spaces is called compact if it maps bounded subsets to relatively compact subsets.} $\iota_{k,k+1} \colon E_k \rightarrow E_{k+1}$. Then the inductive locally convex limit
  \[
     E \coloneq \lim_{\stackrel{\longrightarrow}{k} } E_k
  \]
  exists, i.e.\ is again a Hausdorff locally convex space. In the following we will suppress the index of inductive limits whenever there is no ambiguity.\smallskip
  
  Spaces arising as inductive limits of Banach spaces with compact bonding maps are called \textbf{Silva spaces} or (DFS)-spaces (where DFS stands for (strong) dual \Frechet-Schwartz, as it can be shown that they are precisely the strong duals of \Frechet-Schwartz spaces (see \cite{MR97702}). Infinite-dimensional Silva spaces are not  metrisable yet they have a surprising amount of good topological properties. 
  \begin{rem}[Properties of Silva spaces]\label{rem:pSilva}
    We recall now for the reader's convenience several well-known yet crucial properties of Silva spaces:
    \begin{enumerate}
     \item[(S1)] Silva spaces are \textbf{sequential} \cite[Proposition 6]{MR97702}. This means that the topology is determined by sequences, i.e., sets are closed if and only if they are sequentially closed and functions defined on Silva spaces are continuous if they are sequentially continuous.
     \item[(S2)] The inductive limit is \textbf{compactly regular}, i.e. every compact subset $K \subseteq E$ is already contained in $E_k$ for some $k$ (and compact in the Banach space topology. This entails the following: A sequence $(x_j)_j$ in $E$ converges to $x\in E$ if and only if there exists a fixed $k\in\N$ such that $(x_j)_j$ and $x$ are contained in the Banach space $E_k$ and the sequence converges in the Banach topology to $x$ \cite[Theorem 1]{MR97702}.
     \item[(S3)] In sequential spaces, continuity coincides with sequential continuity. Hence combining 2. and 3., a mapping $\varphi \colon E \rightarrow F$ is \textbf{continuous if and only if $\varphi \circ I_k$ is continuous for every $k$}, where $I_k \colon E_k \rightarrow E$ is the canonical inclusion.
     \item[(S4)]  A version of the Bolzano-Weierstrass holds, i.e. \emph{every bounded sequence in $E$ has a convergent subsequence}. In this case, a sequence $(x_j)_j$ is called bounded if it is bounded as a sequence in one of the $E_k$. Furthermore, every Silva space is \textbf{separable}. (In the present paper we have no need for these properties.)
     \item[(S5)]  Every closed linear subspace of a Silva space is again a Silva space, \cite[Proposition 1]{MR97702}.
    \end{enumerate}
  \end{rem}
  Moreover, Silva spaces are stable under countable inductive limits:
  \begin{la}\label{la:ctbl_Silva}
    A countable inductive limit $F = \underset{\longrightarrow}{\lim}\ F_k$ of Silva space $F_k$ is a Silva space.
  \end{la}
  
  \begin{proof}
    Assume that we have Banach spaces $E^k_\ell$ such that $F_k = \underset{\longrightarrow}{\lim} E^k_\ell$. Applying the usual diagonal argument, we see that the bonding maps $F_k \rightarrow F_{k+1}$ factor through the bonding maps of suitable steps of the limit $F_{k+1}$. However, the bonding maps $E_{\ell}^k \rightarrow E_{\ell'}^k$ are compact operators. Composition of continuous linear mappings with compact operators yield again compact operators, whence the limit factors through a sequence with compact bonding maps. Thus \cite[Lemma 2]{Kom67} (replacing the property ``weakly compact'' with ``compact'') shows that $F$ is again a Silva space.
  \end{proof}
  
  The space of germs of holomorphic functions, which plays a central role in this paper, is a typical example of a Silva space.
  
  \begin{exa}\label{exa:germholo}
    Let $k \in \N$. With respect to the supremum norm, the bounded holomorphic ($\C$-valued) functions 
    $(\BHol (B_{1/n}(\C^k) ),\lVert \cdot \rVert_\infty)$ on the ball $B_{1/n}(\C^k)$ form a Banach space. Shrinking the ball we obtain canonical continuous inclusions of these Banach spaces into each other. This yields an inductive system whose limit is the space of germs of holomorphic functions:
    \[
       \CHolok = \lim_{\longrightarrow} \Hol (B_{1/n}(\C^k)  ) = \lim_{\longrightarrow} \BHol (B_{1/n}(\C^k)  ),
    \]
    Using the fact that restriction to a smaller ball yield compact linear operators (see \cite[Theorem 8.4]{MR1471480}), $\CHolok$ is a Silva space. 
  \end{exa}
  
  \subsection{Meromorphic functions and germs with prescribed linear poles}
  We now  construct a locally convex algebra of meromorphic functions with a prescribed set of  linear poles. 
  
  \textbf{Banach spaces of meromorphic mappings.}
  
  If $P \colon \C^k \rightarrow \C$ is a non-zero polynomial, we define the space
  
  \[
     \Mer_{P}^{\C} (B_{1/n}(\C^k)) \coloneq \left\{\frac{f}{P} \middle| f \in \BHol^\C (B_{1/n}(\C^k))\right\},
  \]
  
  and endow it with the unique Banach space structure making the map $\BHol^\C (B_{1/n}(\C^k)) \rightarrow \Mer_{P}^{\C}(B_{1/n}(\C^k), f \mapsto \frac{f}{P}$ an isometric isomorphism of Banach spaces, i.e.\, we define the norm $\lVert \frac{f}{P}\rVert_{n,P} \coloneq \lVert f\rVert_\infty$.
  \medskip
  
  Fix a countable set $\newLa$ of linear mappings $\C^k \rightarrow \C$ which does not contain the zero-map. For later purposes, we shall always assume that the linear mappings in $\newLa$ \textbf{have real coefficients}, meaning by that they send $\R^k$ to $\R$. We call $\newLa$ a \textbf{generating set} and denote by
  \[
     {\mathcal{S}_\newLa} \coloneq \{L_1 \cdot L_2 \cdots L_m \mid L_1, \ldots , L_m \in {\newLa}, m \in \N \}
  \]
  the semi-group (without unit) generated by the set $\newLa$, which we equip with a natural order by divisibility of polynomials, i.e.\, $P \preceq Q$ if and only if $P=Q$ or $Q = R \cdot P$ for some $R \in \mathcal{S}_{ \newLa}$.
  \begin{rem} The assumption that $\mathcal{S}_{\newLa}$ does not contain $1$ is motivated by the fact that the inverses of the polynomials in $\mathcal{S}_{\newLa}$ will correspond to the poles of the class meromorphic germs  with linear poles in $ \newLa$.
  \end{rem}
  Using the direct product ordering, we obtain a natural (partial) order
  \[
     \left( (n_1,P_1) \leq (n_2,P_2)\right) \Longleftrightarrow \left(	n_1\leq n_2 \text{ and } P_1 \preceq P_2 \right)	\text{ on } \N \times \mathcal{S}_{\newLa}.
  \]
  The meromorphic functions  with linear poles we consider in the examples below arise in quantum field theory from Feynman integrals \cite{Speer} and in number theory from multizeta functions, see e.g.~\cite{MP}. 
  
  \begin{la}\label{la:closed real subspace}
    For all $P \in \mathcal{S}_{\newLa}$, $k,n \in \N$, the sets  
    \begin{align*}
      \Mer_{P}  (B_{1/n}(\C^k)) \coloneq \left\{\frac{f}{P} \in \Mer_{P}^{\C} (B_{1/n}(\C^k)) \middle| f(B_{1/n}(\R^k)) \subseteq \R\right\},\\
      \BHol  (B_{1/n}(\C^k) ) \coloneq \left\{ f \in \BHol^\C(B_{1/n}(\C^k)) \mid f(B_{1/n}(\R^k))\subseteq \R\right\}
    \end{align*}
    are closed real subspaces of $\Mer_{P}^{\C} (B_{1/n}(\C^k))$ and $\BHol^\C (B_{1/n}(\C^k))$, respectively. In particular, both are real Banach spaces and $\BHol (B_{1/n}(\C^k))$ is even a real subalgebra.
  \end{la}
  
  \begin{proof}
    By construction $\Mer_{P}^{\C} (B_{1/n}(\C^k)) \cong \BHol^\C (B_{1/n}(\C^k) )$ via the map $f \mapsto P\cdot f$ where the right hand side carries the topology of uniform convergence. As $P$ has real coefficients, we see that multiplying with $P$ takes $\Mer_{P}  (B_{1/n}(\C^k))$ to the subspace $\BHol  (B_{1/n}(\C^k))$. Hence it suffices to prove the claims for $\BHol^\C(B_{1/n}(\C^k) )$. 
    
    The point evaluations $\ev_z \colon \BHol^\C (B_{1/n}(\C^k) ) \rightarrow \C$  are continuous algebra morphisms for every $z \in B_{1/n}(\C^k)$. As $\R$ is a closed real subalgebra of $\C$, the result then follows from 
    $$\BHol  (B_{1/n}(\C^k)) =\bigcap_{z \in B_{1/n}(\R^k )} \ev_z^{-1}(\R).\qquad \qedhere$$
  \end{proof}
  
  \begin{exa}\label{ex:Feynman}{For any $k\in \N$, we consider the generating set 
    \[
       \newLa^F\coloneq \left\{\C^k \ni (z_1, \cdots, z_k)\mapsto \sum_{i\in I} z_i, \,\, \emptyset\neq I\subseteq [[1, k]]\right\},
    \]
    so that
    \[
       \mathcal{S}_{\newLa}^F \coloneq \left\{\C^k \ni (z_1, \cdots, z_k)\mapsto \prod_{j=1}^J \sum_{i_j\in I_j} z_{i_j}, \quad\emptyset \neq I_j\subseteq [[1, k]]\, \forall j\in [[1, J]] \right \}.
    \]
    This set hosts poles of meromorphic germs which naturally arise from computing Feynman diagrams, hence the choice of superscript $F$.}
  \end{exa}
  \begin{exa} {For any $k\in \N$, we consider the generating set 
    \[
       \newLa^F\supseteq \newLa^C\coloneq \left\{\C^k \ni (z_1, \cdots, z_k)\mapsto \sum_{i=1}^j z_i, \,\,  j\in [[1, k]]\right\},
    \]
    so that
    \[
       \mathcal{S}_{\newLa^C} \coloneq \left\{\C^k \ni (z_1, \cdots, z_k)\mapsto \prod_{j=1}^J \sum_{i_j=1}^j z_{i_j}, \quad j\in [[1, J]],  J\in [[1, k]]\ \right \}.
    \]
    This set hosts poles of meromorphic germs which naturally arise from discrete and integral Laplace transforms on Chen cones $\{(x_1, \cdots, x_k)\in \R_{\geq 0}^k, 0<x_j<\cdots <x_1\}$, hence the choice of superscript $C$, see \cite{GPZ2} for the pole structure of Laplace transforms on convex polyhedral cones.}
  \end{exa}
  \begin{la}\label{lem:cpt-op}
    For all $(n,P) \leq (m,Q) $ in $\N \times \mathcal{S}_{\newLa}$ the inclusion
    \[
       \iota_{(n,P),(m,Q)}^{\C} \colon \Mer_{P}^{\C} (B_{1/n}(\C^k)) \rightarrow \Mer_{Q}^{\C} (B_{1/m}(\C^k)) ,\quad \frac{f}{P} \mapsto \frac{\left.f\cdot \frac{Q}{P}\right|_{B_{1/m}(\C^k)}}{Q},
    \]
    is continuous linear and if $n < m$, the inclusion is a compact operator. Moreover, the inclusion restricts to a continuous linear and compact operator $\iota_{(n,P),(m,Q)}  \colon \Mer_{P} (B_{1/n}(\C^k)) \rightarrow \Mer_{Q}   (B_{1/m}(\C^k))$.
  \end{la}
  
  \begin{proof}
    By construction we have $\Mer_{P}^{\C} (B_{1/n}(\C^k)) \cong \BHol^\C (B_{1/n}(\C^k) )$. Since multiplication turns $\BHol^\C (B_{1/n}(\C^k) )$ into a Banach algebra and $P$ divides $Q$, we see that the map $\Mer_{P}^ {\C} (B_{1/n}(\C^k))\ni \frac{f}{P} \mapsto f\cdot \frac{Q}{P} \in \BHol^\C(B_{1/n}(\C^k))$ is continuous linear. It is well known that the composition of a continuous linear map and a compact operator is again a compact operator.  It is a known fact that the inclusion maps
    \[
       \BHol^\C (B_{1/n}(\C^k) ) \rightarrow \BHol^\C (B_{1/m}(\C^k) ),\ f \mapsto f|_{B_{1/m}(\C^k)}
    \]
    are continuous linear and  compact operators for $n < m$  cf.\, \cite[Appendix A]{MR3453778} or \cite[Section 8]{MR1471480}. Now since the elements in $\mathcal{S}_\newLa$ have real coefficients, $\iota_{(n,P),(m,Q)}^ {\C}$ takes $\Mer_{P}  (B_{1/n}(\C^k))$ to $\Mer_{Q}  (B_{1/m}(\C^k))$, whence its restriction $\iota_{(n,P),(m,Q)}  \colon \Mer_{P} (B_{1/n}(\C^k)) \rightarrow \Mer_{Q} (B_{1/m}(\C^k))$ becomes a continuous linear and compact operator between real Banach spaces.
  \end{proof}
  
  The upshot of Lemma \ref{lem:cpt-op} is that for each $k \in \N$ we obtain an inductive system of (real or complex) Banach spaces 
  with continuous linear connecting maps (the next picture shows a piece of the inductive system for $L,P\in\newLa$, (we are most interested in the real system shown below but the statement also holds for the complex Banach spaces):
  \[
     \adjustbox{scale=.8,center}{
     \begin{tikzcd}[thick,scale=0.6, every node/.style={transform shape}]
       \Mer_{1} (B_1(\C^k)) \arrow[r] \ar[d] & \Mer_{L} (B_1(\C^k)) \ar[r] \ar[d] 	& \Mer_{LP} (B_1(\C^k)) \ar[r] \ar[d]	& \cdots \\
       \Mer_{1} (B_{1/2}(\C^k)) \arrow[r] \ar[d] & \Mer_{L} (B_{1/2}(\C^k)) \ar[r] \ar[d]	& \Mer_{LP} (B_{1/2}(\C^k)) \ar[r] \ar[d] 	& \cdots \\
       \vdots 							 & \vdots 						& \vdots 						& \ddots
     \end{tikzcd}
     }
  \]
  where every arrow pointing down represents a compact operator.

  We are now ready to give a more precise description of the locally convex algebra $\CMerok$ of germs of meromorphic functions with poles in $\newLa$. Its real subalgebra $\RMerok$ of germs mapping $\R^k$ to $\R$ will serve as a basic building block for the topological decomposition we are about to construct.
  Note that the algebra of germs of meromorphic functions contains the space $\CHolok$ (resp. $\RHolok$ for the real subalgebra) of germs of holomorphic functions which form a subalgebra of $\CMerok$.   
  \begin{prop}\label{prop:lcxv-alg}
    The locally convex inductive limits 
    \begin{align*}
      \CMerok &\coloneq \lim_{\substack{\longrightarrow \\ (n, P)\in\N\times \mathcal{S}_{\newLa}}} \Mer_{P}^{\C} (B_{1/n}(\C^k))\\
      \RMerok &\coloneq \lim_{\substack{\longrightarrow \\ (n, P)\in\N\times \mathcal{S}_{\newLa}}} \Mer_{P}  (B_{1/n}(\C^k))
    \end{align*}
    are Hausdorff and complete. Moreover, they are Silva spaces and the pointwise multiplication turns both into locally convex algebras.
    
    The spaces $\CHolok$ of germs of holomorphic functions and $\RHolok \coloneq \lim\limits_{\substack{\longrightarrow \\ n}} \BHol  (B_{1/n}(\C^k))$ are Silva space which continuously inject into $\CMerok$ and $\RMerok$, respectively. 
  \end{prop}
  
  \begin{proof}
    We prove the complex case and note that the statements for $\RHolok$ and $\RMerok$ can be proved by the same argument after replacing all the complex spaces with the corresponding real subspaces. Since the directed set $\N\times\mathcal{S}_{\newLa}$ is countable (and without a maximum), the inductive system $(\Mer_{P}^{\C} (B_{1/n}(\C^k)))_{(n,P) \in \N \times \mathcal{S}}$ admits a cofinal sequence of Banach spaces with compact connection maps. It follows that  the locally convex topology of the inductive limit is Hausdorff, complete and a Silva space. Observe that the multiplication in $\CMerok$ factors through multiplication maps on the steps of the limit (using that the product commutes with the inductive limit, \cite[Theorem 9]{Kom67}), i.e.\ for $n<m$:
    \begin{align*}
      m_{(n,P),(m,Q)}^{\C} \colon \Mer_{P}^{\C} (B_{1/n}(\C^k)) \times &\Mer_{Q}^{\C} (B_{1/m}(\C^k)) \rightarrow \Mer_{PQ}^{\C} (B_{1/m}(\C^k),\\ \left( \frac{f}{P} , \frac{g}{Q}\right) &\mapsto \frac{f|_{B_{1/m}(\C^k)}\cdot g|_{B_{1/m}(\C^k)}}{PQ}.
    \end{align*}
    Using  again that fact that $\Mer_{P}^{\C} (B_{1/n}(\C^k)) \cong \BHol^\C (B_{1/n}(\C^k))$ and that the latter spaces are Banach algebras with respect to pointwise multiplication, we see that the bilinear map $m^{\C}_{(n,P),(m,Q)}$ is continuous on every Banach step generating the inductive limit. Since $\CMerok$ is a Silva space, this  implies that multiplication is continuous, see \cite[Proposition 6 and Theorem 1]{MR97702}. We conclude that $\CMerok$ is a locally convex algebra. 
    
    As for the structure of the space of holomorphic germs, Example \ref{exa:germholo} shows that
    \[
       \CHolok=   \underset{\substack{\longrightarrow \\ n\in\N }}{\lim}\BHol^\C (B_{1/n}(\C^k) )
    \]
    is a Silva space. For the space $\CHolok$ we can leverage that the bonding maps of the inductive limit restrict to the bonding maps of the real inductive system formed by the $\BHol  (B_{1/n}(\C^k) )$. Thus  $\RHolok$ is also a Silva space.
    By the universal property of the direct limit, the continuity of the inclusions $\BHol^\C (B_{1/n}(\C^k)) \hookrightarrow  \Mer_{P}^{\C} (B_{1/n}(\C^k))$ for any $n$ in $\N$ and any $P\in \mathcal{S}_{\newLa}$, gives rise to a continuous embedding $\CHolok \rightarrow \CMerok$.
  \end{proof}
  
  \begin{rem}
    We will see later (Theorem \ref{thm:splittingspaces}) that the canonical inclusion map from $\RHolok$ to $\RMerok$ is not only continuous but moreso, a topological embedding.
  \end{rem}

  The next statement, which is useful for the sequel, shows how convergence in $\CMerok$ relates with uniform convergence. Again an analogous statement holds for the real subspace $\RMerok$.
  
  \begin{cor}  \label{cor:silva_and_uniform_convergence}
    Let $k\in\N$ be fixed and let $(f_m)_{m\in\N}$ be a sequence in
    \[
       {\displaystyle\CMerok= \underset{\substack{\longrightarrow \\ (n,P)}}{\lim} \Mer_P^{\C}(B_{1/n}(\C^k))}
    \]
    that converges to $f\in \CMerok$. Then there exists an open ball $U\subseteq \C^k\setminus\{0\}$ on which $f$ and each $f_m$ are bounded and holomorphic, and such that the restrictions $f_m|_{\overline U}$ converge uniformly to $f|_{\overline U}$.
  \end{cor}
  
  \begin{proof}
    First of all, since $\CMerok$ is a Silva space, this means that there is one fixed number $n\in\N$ and a polynomial $P\in \mathcal{S}_{\newLa}$ such that the whole sequence $(f_m)$ as well as the limit $f$ lie in the space $\Mer_P^{\C}(B_{1/n}(\C^k))$ and we have that $(f_m)_m$ converges in $\Mer_P^{\C}(B_{1/n}(\C^k))$ to $f$.
    By the definition of the topology of $\Mer_P^{\C}(B_{1/n}(\C^k))$, this means that the sequence $(P f_n)_n$ converges uniformly to $Pf$ on $B_{1/n}(\C^k)$. 
    
    Since the zero-set $P^{-1}(\{0\})$ of the polynomial $P$ is a finite union of proper vector subspaces, it does not contain interior points. Hence, the open set
    $W \coloneq B_{1/n}(\C^k)\setminus P^{-1}(\{0\})$ is non empty.
    Now, we pick a point $a\in W$ and an $\varepsilon>0$ such that the compact ball of radius $\varepsilon$ around $a$ lies in 
    $W$. 
    The claim then follows for the open ball $U$ of radius $\varepsilon$ around $a$.
  \end{proof}

  \subsection{The subspaces of polar germs}
  
  We now sharpen  the decompostion of  the space $\CMerok$  as a (non-direct) sum of holomorphic germs and meromorphic germs to a topological decomposition.
  For this, it is essential to work with the real subspace $\RMerok$. Our aim is to construct a subspace of $\RMerok$ consisting of the so called polar germs. To define polar germs we need some preparatory definitions taken from \cite{CGPZ}.
  
  \begin{defn}
    Let $f$ be a meromorphic function defined on an open connected subset $U$ of $\C^k$ for a fixed $k\in\N$. If there are linear forms $L_1, \ldots , L_n$ on $\C^k$ and a meromorphic function $g$ on an open connected subset $W$ of $\C^n$, such that $f=g\circ\phi$ on $U\cap \phi ^{-1}(W)$, where $\phi=(L_1, \cdots, L_n)\colon \C ^k \to \C ^n$. We then say that 
    \begin{itemize}
     \item $f$ \textbf{depends} on the linear subspace $\mathrm{span}(L_1, \cdots, L_n)$. 
     \item  One can show that there is a smallest subspace with this property generated by some linear forms $L_1,\cdots, L_n$ . We call the smallest linear subspace on which $f$ depends, $\mathrm{Dep} (f) = \mathrm{span}(L_1, \cdots, L_n) \subseteq(\C ^k)^*$, the \textbf{dependence subspace} of $f$ \cite[Definitions 2.9 and 2.13]{CGPZ}.
    \end{itemize}
    There is also a dual notion: 
    \begin{itemize}
     \item We say that $f$ is \textbf{constant in direction} $v\in\C^k$ if the directional derivative $D_vf$ is equal to the constant zero function.
     \item The \textbf{independence subspace} of $f\colon U\to\C$ is defined as the set of all vectors such that $f$ is constant in the direction of $v$:
      \begin{equation}\label{eq:indspace}
        \Indep(f)\coloneq \{v\in\C^k \mid D_v(f)=0\},
      \end{equation}
    \end{itemize}
  \end{defn}
  
  These notions are connected via the observation:
  \[
     \Dep(f) = \{L\in(\C^k)^* \mid L|_{\Indep(f)}=0\}.
  \]
  
  Note that derivatives for meromorphic functions only need to be calculated at the regular, holomorphic points as they form an open dense connected subset of the domain of the original function. 
  For a meromorphic germ $f\in\RMerok$ the spaces
  $\Dep(f)\subseteq(\C ^k)^*$ and $\Indep(f)\subseteq \C^k$ are defined using any of its representing {locally defined functions}. It is clearly well-defined. 
  Furthermore, the notion of dependence subspace is compatible with the Silva-space topology via the following statement:
  
  \begin{la} \label{la:dep_and_silva_convergence}
    Fix $k\in\N$ and assume that $f_1,f_2,\ldots $ is a sequence of elements in $\RMerok$ which converges to $f$ in $\RMerok$. Then for any $\ell$ in $\Dep(f)$, there is a subsequence $(f_{m_j})_{j\in \N}$ and a sequence $\ell_{m_j}$ in $ \Dep(f_{m_j}), j\in \N$ which converges to $\ell$. 
  \end{la}
  We postpone the rather technical proof to Appendix \ref{App:proofdetails}.
  To obtain a splitting of the space of meromorphic germs into a direct sum of holomorphic germs and a topological vector space complement, we need the following definitions taken from \cite[Definition 3.1 and Definition 3.2]{GPZ}.
  
  \begin{defn}\label{defn:suppcones}
    The following definitions depend on the choice of the inner product $Q$ and will thus be labeled with $Q$.
    \begin{itemize}
     \item Two meromorphic germs $f$ and $h$ are called \textbf{orthogonal} if $\Dep (f) \perp^Q \Dep (h)$. 
     \item We call a meromorphic germ \textbf{polar germ}, if it can be written in the form $\frac{h}{P}$ for $h \in \RHolok$ and $P \in \mathcal{S}_\newLa$ such that $h$ and $P$ are orthogonal.
     \item We denote by $\polarMerok$ the linear subspace of $\RMerok$ \textbf{generated by the polar germs}.
     \item The cone generated by the linear forms $L_1, \cdots, L_{\ell_i}$ arising in the polar germ $g :=\frac{h }{L_1^{m_{1}}\cdots L_{\ell}^{m_{\ell} }}$  is called the \textbf{supporting cone} of $g$.
     \item A family of polar germs is called \textbf{properly positioned} if there is a choice of a supporting cone for each of the polar germs such that the resulting family of cones is properly positioned, i.e.\ if the cones meet along faces and the union does not contain any nonzero linear subspace. 
    \end{itemize}
  \end{defn}
  
  \begin{thm}\label{polargerms_closed}
    The vector space $\polarMerok$ is closed in $\RMerok$.
  \end{thm}
  
  \begin{proof}
    From \Cref{prop:lcxv-alg} we know that $\RMerok$ is a Silva space and therefore, its topology is sequential. By \Cref{rem:pSilva} (S1), it suffices to check closedness of a subset using sequences. To this end, we consider     a sequence $f_1,f_2,\ldots$ in $\polarMerok$ converging in $\RMerok$ to a meromorphic germ $f$. It remains to show that $f$ lies  in $\polarMerok$, i.e. that $f$ is a linear combination of polar germs.
    
    Since $f$ can be written as a sum of polar germs and holomorphic germs (cf.\ \cite[Corollary 4.16]{GPZ}), we subtract the polar germs from every $(f_n)_n$ and the limit $f$. Hence without loss of generality $f$ is a purely holomorphic germ, i.e.\ $f \in \RHolok$. It remains to show that $f=0$.
    
    By \Cref{rem:pSilva} (S2), the converging sequence $(f_n)_n$ lies entirely in one of the Banach steps $\Mer_P(B_{1/N}(\C^k))$ for an $N\in\N$ and a $P=L_1^{M_1}\cdots L_\ell^{M_\ell}\in S_\newLa$ for a finite set of distinct linear forms $\{ L_1,\ldots,L_\ell \}\subseteq \newLa$ and exponents $M_1,\ldots,M_\ell\in\N$ such that the denominator of each $f_n$ divides $P$.
    
    Applying again the algebraic splitting \cite[Corollary 4.16.]{GPZ} into holomorphic and polar parts, every $f_n$ can be decomposed into a holomorphic part and a finite sum of polar parts. By assumption each holomorphic part is trivial.
    So, for each $n\in\N$ we have
    \[
       f_n = \sum_{i=1}^{p_n} \frac{h_i^{(n)}}{L_1^{m_{i,1}^{(n)}}\cdots L_{\ell_i}^{m_{i,\ell_i}^{(n)}}}
    \]
    with each $h_i^{(n)}$ in $\RHolok$. 
    
    A priori the number $p_n$ of summands can depend on the  element in the sequence but the numbers $\{p_n | n\in\N\}$ are bounded above since the given polynomial $P$ can only have a finite number of decompositions. We may therefore pass to a subsequence and assume that $p_n = p$ is independent of $n$. Similarly, we pass to another subsequence and can achieve that the exponents $m_{i,j}^{(n)}$ do no longer depend on $n$: 
    \begin{equation}\label{eq:useful}
      f_n = \sum_{i=1}^{p} \frac{h_i^{(n)}}{L_1^{m_{i,1}}\cdots L_{\ell_i}^{m_{i,\ell_i} }}\,.
    \end{equation}
    Eq. (\ref{eq:useful}) gives   a  decomposition of $f_n$ into a sum of polar germs, supported by  a family of supporting cones (Definition \ref{defn:suppcones}), which is independent of $n$. From \cite[Lemma 4.10 and Lemma 4.11]{GPZ} we pick a cone subdivision,  and derive from this decomposition in polar germs, another decomposition in polar germs whose supporting cones are properly positioned \cite[Definition 3.2.]{GPZ}. By this we mean that the cones arising in the underlying family of supporting cones  
    meet along faces and their union does not contain any nonzero linear subspace
    (equivalently, any line).  Since the construction of this properly positioned family only depends on the family of supporting cones which is independent of $n$, the same subdivision can be implemented on all $f_n's$. Taylor expanding the $h_i^{(n)}$ with respect to an orthogonal basis containing the $L_{i}$ (cf. \cite[Theorem 2.10 and Corollary 4.18]{GPZ}), we may assume w.l.o.g. that the polar germs $\frac{h_i^{(n)}}{L_1^{m_{i,1}}\cdots L_{\ell_i}^{m_{i,\ell_i} }}, i=1, \ldots, p$ are properly positioned. Note that this involves changing both the numerator and denominator functions of the every summand as well as passing to another $p$. We suppress this in the notation. A priori the new data depends on $n$, but since there are again only finitely many possibilities for the polynomials in the $L_i$, and for $p$, we may again assume that these items do not depend on $n$.      
    \newcommand{\ballAround}[2]{#1+B_{#2}(\C^k)}
    
    By Corollary \ref{cor:silva_and_uniform_convergence} we can find an $a$ in $\C^k\setminus\{0\}$ and an $R>0$ such that 
    each $f_n$ is holomorphic on
    the open ball $U:=\ballAround{a}{R}\subseteq\C^k\setminus\{0\}$ and such that $f_n|_U$  converges uniformly to $f$ on $U$.
    
    There are two cases to consider: 
    Either there exists a smaller open ball 
    $V:=\ballAround{a}{r}$ inside $U$ such that all numerator sequences 
    \[
       (h_1^{(n)})_n, (h_2^{(n)})_n,\ldots,(h_p^{(n)})_n
    \]
    are bounded on $V$, or for every smaller open ball $V$ there is at least one $i\in\{1,\ldots,p\}$ such that the sequence $(h_i^{(n)})_n$ is unbounded on $V$. \medskip
    
    \textbf{Case 1: There is an open ball $V\subseteq U$ such that all numerator sequences are bounded on $V$:}\\
    If all the $p$ numerator sequences $(h_1^{(n)})_n, (h_2^{(n)})_n,\ldots,(h_p^{(n)})_n$ are bounded on $V$, we apply Montel's theorem (see \cite[Theorem 14.6]{Rud87} for the one-dimensional statement whose proof generalises to holomorphic functions of several variables): There exists a subsequence such that each $(h_i^{(n)})_n$ converges to some holomorphic function $h_i$ uniformly on compacts subsets of $V$. Hence, we have the following equality on $V$:
    \begin{align}\label{frac:rep}
      f = \sum_{i=1}^{p} \frac{h_i}{L_1^{m_{i,1}}\cdots        			L_{\ell_i}^{m_{i,\ell_i} }}
    \end{align}
    We will now show that this sum is in fact equal to zero. Note that by construction $\frac{h_i^{(n)}}{L_1^{m_{i,1}}\cdots  L_{\ell_i}^{m_{i,\ell_i} }} \rightarrow \frac{h_i}{L_1^{m_{i,1}}\cdots  L_{\ell_i}^{m_{i,\ell_i} }}$ uniformly on $V$ for $i=1,\ldots, p$. We can thus invoke Lemma  \ref{la:dep_and_silva_convergence} (indeed we need only Step 2 and Step 3 3 from its proof in Appendix \ref{App:proofdetails}) to see that every $\ell$ in $\Dep (h_i)$ is the limit of a sequence of elements in $\Dep (h_i^{(n)})$. Now since for every $i$, the summand $h_i^{(n)} /L_1^{m_{i,1}}\cdots  L_{\ell_i}^{m_{i,\ell_i}}$ is polar, the dependent subspaces of the $h_i^{(n)}$ are $Q$-orthogonal to the dependent subspace of the denominator. As orthogonal complements are closed, we deduce that  $\ell$ is also orthogonal to this dependence subspace. In other words, every summand $\frac{h_i}{L_1^{m_{i,1}}\cdots        	L_{\ell_i}^{m_{i,\ell_i} }}$ is a polar germ, so the right hand side of \eqref{frac:rep} is a sum of polar germs  and the ``locality'' lemma in \cite[Lemma 3.5]{GPZ} (with $a_i=1$) then yields $f=0$.\medskip
    
    \textbf{Case 2: For each open ball $V\subseteq U$ at least one $(h_i^{(n)})_n$ is unbounded on $V$:}\\
    
    By assumption, for each radius $R'<R$, on each $\ballAround{a}{R'}$ at least one of the $h_i^{(n)}$ is unbounded. Since there are only finitely many values for $i$, we may assume that there is one $i_0$ such that $(h_{i_0}^{(n)})_n$ is unbounded on every $\ballAround{a}{R'}$. To simplify notation, we may assume that $i_0=1$.
    
    Now we fix a sequence of positive radii $R=R_1>R_2>R_3>\cdots$ converging to $0$ and denote by $V_n := \ballAround{a}{R_n}$ the corresponding open balls.
    After passing---once again---to a subsequence, we may assume that $r_n:=\| h_1^{(n)}|_{V_n}\|_\infty>n$ for each $n$ in $\N$.
    We can now divide the whole equation by $r_n$ and obtain:
    \[
       \frac{1}{r_n} f_n = \sum_{i=1}^{p} \frac{h_i^{(n)}/r_n}{L_1^{m_{i,1}}\cdots L_\ell^{m_{i,\ell}}}\,.
    \]
    Now, the new numerator sequence $(h_1^{(n)}/r_n)_n$ is bounded on each small ball $V_n$, but there is no ball $V_n$ where it converges to zero since it has norm $\Vert h_i^{(n)}|_{V_n}\Vert /r_n=1$.
    
    If on each open ball, at least one of the new numerator sequences $(h_i^{(n)}/r_n)_n$ is unbounded, we repeat this procedure until we end up with a case where all numerator sequences are bounded on each open ball and at least one of them does not converge to $0$ on every open ball.
    
    Now, as in the above Case 1, we use Montel's theorem and after passing to a subsequence, we may assume that each $(h_i^{(n)})_n$ converges uniformly on compact ball around $a$ to a holomorphic functions $h_i$. Therefore, we have:
    \[
       0 = \sum_{i=1}^{p} \frac{h_i}{L_1^{m_{i,1}}\cdots L_\ell^{m_{i,\ell}}}\,.
    \]
    By construction, at least one of the numerator functions is not the zero function. Now we exploit the fact that the polar germs are properly positioned. Thus the uniqueness of the Laurent decomposition  \cite[Corollary 3.8]{GPZ} (here ) shows that the new numerator sequences have to converge to $0$, but this contradicts the fact that there is at least one numerator sequence not converging to zero. This contradiction shows that Case 2 cannot happen.
  \end{proof}
  
  \begin{rem}
    We observe that the implementation of the ``locality'' lemma in Case 1 of the proof of Theorem \ref{polargerms_closed} hinges on the fact that we are working with the real subspace $\RMerok$. The cited argument does not hold for the full complex space $\CMerok$. This is the reason why we need  to restrict to the real subalgebra.
  \end{rem}

  \begin{thm}\label{thm:splittingspaces}
    The space $\RMerok$ splits as a topological direct sum of the holomorphic germs and the space generated by the polar germs. In other words, the map 
    \[
       \sigma \colon \RHolok \times \polarMerok \rightarrow \RMerok, \quad (f,g) \mapsto f +g
    \]
    is an isomorphism of locally convex spaces. In particular, the Silva topology of $\RHolok$ coincides with the subspace topology induced by $\RMerok$.
  \end{thm}
  
  \begin{proof} 
    Recall from \Cref{prop:lcxv-alg} that the canonical inclusion of $\RHolok$ into $\RMerok$ is continuous. Since $\RMerok$ carries the subspace topology this shows that the mapping $\sigma$ is continuous. It is a bijection since we know from \cite[Theorem 4.4]{GPZ} that the two spaces form an algebraic direct sum. In Theorem \ref{polargerms_closed} we have seen that $\RMerok$ is a closed subspace of $\RMerok$, whence it is also a Silva space by \cite[Proposition 1]{MR97702}. As the product of two Silva spaces is again a Silva space, we see that $\sigma$ defines a continuous bijection between Silva spaces. Every Silva space is webbed and ultrabornological as a countable inductive limit of Banach spaces. Thus we can apply the open mapping theorem \cite[24.30]{MaV97} to infer that $\sigma$ is also open, whence an isomorphism of locally convex spaces.
  \end{proof}
  
  We established a topological splitting for multivariate germs of meromorphic functions into a direct sum of the germs of holomorphic functions and polar germs. Before we enhance this construction via another limiting process to $\C^\N$ it is worthwhile to consider related results for the one-dimensional case.
  
  \subsection{Germs of meromorphic functions in one variable}  
  In this section we  consider the differences between the multidimensional construction and the case of germs in one variable. First of all, the algebra of germs of meromorphic functions in one variable is known to be a Silva algebra, \cite[Example 3]{MR3883648}.
  To see this we will consider the generating set of linear forms $\newLa= \{\id \colon \C \rightarrow \C\}$ (or to be more in line with the notation of loc.cit., poles will be of the form $z^\ell$ for some $\ell \in \N$). Hence for every monomial $z^\ell, \ell \in \N$ we have the inclusion $\BHol^\C (B_{1/n}(\C)) \subseteq \Mer_{z^\ell}^\C (B_{1/n}(\C))$ and the latter space contains all meromorphic functions with pole of degree at most $\ell$ in $0$. Equipping both spaces with the Banach space structure induced by the supremum norm, the \textbf{space of meromorphic germs at}  $0 \in \C$
  \begin{displaymath}
    \Mero^\C(\C) = \{f \colon U \rightarrow \C\cup\{\infty\} \text{ meromorphic and } 0 \in U \subseteq \C \text{ open}\}/\sim ,
  \end{displaymath}
  becomes a Silva space (cf.\ also \cite[Theorem 8.4]{MR1471480}) and one concludes as above that $\Mero^\C(\C)$ is a locally convex algebra.
  
  In the one-dimensional case, the complement to the holomorphic germs admits an attractive explicit description which we now recall. Consider the space of polynomials $\Poly^\infty(X)\coloneq \C[X]=\mathrm{span}\{X^0,X^1,X^2,\ldots\}$ in a formal variable $X$ and denote by $\Poly^\infty_*(X)\coloneq \text{ span}\{X^1,X^2,\ldots\}$ the linear subspace of polynomials without constant term.
  This space is a Silva space as the direct union $\Poly^\infty_*(X)=\bigcup_{n\in\N} \Poly^n_*(X)$ of finite-dimensional spaces $\Poly^n_*(X)\coloneq \mathrm{span}(X^1,\ldots, X^n)$.
  The bonding maps $\Poly^n_*(X)\to \Poly^{m}_*(X)$ are compact operators for $m\geq n$ since the corresponding spaces are finite-dimensional. 
  Now $\Mer_{z^\ell}^\C(B_{1/n}(\C)) = \BHol^\C (B_{1/n}(\C)) \oplus \Poly^n_*\left(\frac{1}{z}\right)$ as locally convex spaces. Since inductive limits and (finite) products of locally convex spaces commute, this shows that
  \begin{align*}
    \Mero^\C(\C) &= \lim_{\substack{\longrightarrow \\ n,\ell}} \Mer_{z^\ell}^\C (B_{1/n}(\C)) = \lim_{\substack{\longrightarrow \\ n}} \BHol^\C (B_{1/n}(\C)) \oplus \Poly^n_*\left(\frac{1}{z}\right) \\ &= \Holo^\C(\C) \oplus \Poly^\infty_{\ast} \left(\frac{1}{z}\right).
  \end{align*}
  The construction relies on the fact that every meromorphic function can be uniquely split as a holomorphic function plus a polar part, here a polynomial in $1/z$ without constant term. This identification with the algebra of polynomials explains why the one-dimensional case is much simpler and does not need to reference additional structures (such as an inner product).
  
  \begin{rem}
    The topology on germs of meromorphic functions in one variable was also constructed in \cite{MR3883648} and \cite{GrosseErdmann}. The latter source also provides   a topology on spaces of meromorphic functions (rather than just germs of such functions).
    For meromorphic functions the multiplication turns out to be only separately continuous (see \cite[Theorem 5]{GrosseErdmann}) and so these spaces are not locally convex algebras.
    
    It should be noted that although $\Mero^\C(\C)$ is algebraically a field and topologically a locally convex algebra, inversion is \emph{not} continuous with respect to the topology just described, 
    hence $\Mero^\C(\C)$ is not a topological field as can be seen on the following simple counterexample.
    \begin{coexa} 
      The polynomial sequence $(z-1/m)_m$ converges to $z$, yet the inverse sequence  $\left((z-1/m)^{-1}\right)_m$ diverges, since it is not contained in any single step $\Mer_{z^\ell}^\C (B_{1/n}(\C)), n, \ell \in \N$ of the inductive limit $ \Mero^\C(\C)$.
    \end{coexa}This  comes as no surprise  since there is no complex locally convex division algebras outside $\C$ according to the (locally convex) Gelfand-Mazur Theorem (see e.g.~\cite[Remark 4.15]{MR1948922} or \cite[Theorem 1]{ArensLinearTopologicalDivisionAlgebras}).
  \end{rem}
  
  There are more differences between the one-dimensional and the multivariate case which are worth mentioning. This is already relevant  at the level of the Banach steps used to construct the inductive limit leading to the Silva topology. To showcase these differences we begin with an easy result from complex analysis.
  
  \begin{la}\label{la:oneD}
    Let $R>0$, then the map 
    \[
       m_R\colon \BHol^\C (B_R(\C)) \rightarrow \BHol^\C (B_R(\C)),\quad f \mapsto (z \mapsto zf(z)),
    \]
    is a topological embedding as a consequence of 
    \begin{equation}\label{Restimate}\lVert m_R(f)\rVert_\infty = R \lVert f\rVert_\infty.
    \end{equation}
  \end{la}
  
  \begin{proof}
    We use the maximum principle for holomorphic functions on $B_R(\C)$:
    \begin{align*}
      \lVert m_R(f)\rVert_\infty &= \sup_{|z|<R} |zf(z)|= \lim_{r \rightarrow R} \sup_{|z|=r}|z||f(z)| = R \lim_{r \rightarrow R} \sup_{|z|=r}|f(z)| \notag \\ &= R \lVert f\rVert_\infty,
    \end{align*}
    from which it follows that $m_R$ is continuous and a homeomorphism onto its image. 
  \end{proof}
  
  As stated in  Lemma \ref{la:oneD}, multiplying with the variable $z$ (i.e.\ with the linear form $\id \colon \C \rightarrow \C$), we obtain a topopological equivalence of the space of bounded holomorphic functions on the ball. This fails in the multivariate case as the following example shows:
  
  \begin{exa}\label{ex:counter}
    Let $D \coloneq \{(z,w) \in \C^2 \mid |z|^2+|w|^2 <1\}$ be the open unit ball in $\C^2$. We consider the holomorphic map 
    \[
       g\colon D \rightarrow \C, \quad g(z,w) \coloneq \frac{1}{\sqrt{1-w}}.
    \]
    Now $g$ has a singularity at the boundary point $(0,1)$, but it is not hard to see that $f (z, w)\coloneq z\, g(z,w)$ is continuous and bounded on all of $\overline{D}$.
  \end{exa}
  
  The pathology encountered in Example \ref{ex:counter} is a singularity located on the boundary of the open ball we wish to consider, a problem one could avoid by restricting to a smaller ball. This  hints to the added difficulty induced by   considering spaces of meromorphic germs in several variables. However, as our next result shows, it is actually sufficient to shrink the ball to obtain a well behaved splitting operation between spaces of bounded holomorphic mappings:

  \begin{la}\label{lem:splittingoff}
    Given $f\in \BHol^\C (B_{1/n}(\C^k)) $, $L \in \newLa$ and $n \in \N$ there is an integer $m > n$ such that the maps $h:= f\circ \pr_L|_{B_{1/m}(\C^k)} $ and $g=\left.\left(f- f\circ\pr_L\right) / L\right|_{B_{1/m}(\C^k)}$ are bounded holomorphic functions on $B_{1/m}(\C^k)$ and the associated mapping  
    \[
       \theta_{n,m,Q}^L \colon \BHol^\C(B_{1/n}(\C^k)) \rightarrow \BHol^\C (B_{1/m}(\C^k))^2,\quad f= L\cdot g + h \mapsto (g,h),
    \]
    is continuous linear.
  \end{la}
  
  The proof for this result is a long computational argument which we relegate to Appendix \ref{App:proofdetails}. We shall not use this result in the following, though it is certainly of independent interest.

 \section{Meromorphic germs on $\C^\N$ as a Silva space}
  
  In previous sections we have fixed the dimension $k \in \N$ and constructed spaces of germs of  holomorphic functions and meromorphic with poles determined by a certain generating set $\newLa$. 
  
  We will now consider the (inductive) limit of the spaces $\RMerok$ for $k \rightarrow \infty$.
  Note first that the projection $\text{ pr}_k\colon \C^{k+1} \rightarrow \C^k$ onto the first $k$ components yields a projective system of bonding maps which describe the locally convex direct product $\C^\N$.
  Denote the canonical projection onto the $k$-th step by $\pi^\N_k \colon \C^\N \rightarrow \C^k$.
  Composition with the bonding maps of the projective limit yields continuous linear maps $\text{ pr}_k^\ast\colon \RMerok \rightarrow \Mero_\newLa(\C^{k+1}), g \mapsto g \circ \text{ pr}_k$. 
  Moreover, we have to pick a system of continuous linear polar parts compatible with the the projection.
  \begin{defn}
    Choose and fix for each $k \in \N$ a generating set $\newLa_k$ which is at most countable and 
    \begin{enumerate}
     \item which consists of non-zero linear mappings $\C^k \rightarrow \C$ with real coefficients,
     \item and such that for every $k \in \N$ and $L \in\newLa_k$ we have that $\text{ pr}_k^* L \in\newLa_{k+1}$.
    \end{enumerate}
    We will use the notation $\BnewLa := \bigcup_{k \in \N}\newLa_k$ and ${\mathcal S}_{\BnewLa}:= \bigcup_{k \in \N}{\mathcal S}_{\newLa_k}$.
    To simplify notation, we will write $\RMerok$ for $\Mero_{\newLa_k}(\C^k)$.
  \end{defn}
  
  With these choices we obtain an inductive system of locally convex spaces which allows us to construct an (locally convex) inductive limit 
  \begin{align*}
    \RMeroN &= \{f \colon \C^\N \rightarrow \C \mid \exists k \in \N, f_k \in \RMerok\, \text{ such that } f = f_k \circ \pi_k^\N \}\\
    &=\lim_{\substack{\longrightarrow \\ k}} \RMerok,
  \end{align*}
  which consists of germs of meromorphic germs with poles in ${\mathcal S}_\newLa$.

  The next result shows that holomorphic functions on the direct product $\C^\N$ factor through some holomorphic function on $\C^k$.
  
  \begin{la}\label{lem:dimlim:holo}
    Let $W \subseteq \C^\N$ be an open $0$-neighborhood and $f \colon W \rightarrow \C$ be holomorphic. Then there exists $k \in \N$ and an open $0$-neighbourhood $U \subseteq \C^k$ such that $f|_{\pi_k^{-1}(U)} = \tilde{f} \circ \pi_k^{\N}|_{(\pi_k^{\N})^{-1}(U)}$, where $\tilde{f} \colon U \rightarrow \C$ is holomorphic and $\pi_k^{\N} \colon \C^\N \rightarrow \C^k$ is the canonical projection (onto the first $k$ components).
    In particular, every germ of a holomorphic function on $\C^\N$ coincides with the germ of a holomorphic function on some $\C^k$.
  \end{la}
  
  \begin{proof}We simplify the notation by setting $\pi_k:=\pi_k^\N$.
    
    Consider a holomorphic function $f \colon W \rightarrow \C$ and fix $R > 0$ such that $|f(0)|<R$. By continuity, $f^{-1}(B_R(0))$ is an open $0$-neighbourhood. We can pick $k \in \N$ and a zero-neighborhood $U = U_1 \times \cdots \times U_k \subseteq \C^k$ which satisfies $\pi_k^{-1}(U) = U_1 \times \cdots \times U_k \times \prod_{n>k} \C \subseteq f^{-1}(B_R (0))$. Consider now the holomorphic map 
    \[
       \tilde{f} \colon U_1 \times \cdots \times U_k \rightarrow \C,\quad (z_1,\ldots,z_k) \mapsto f(z_1,\ldots,z_k , 0, 0,\ldots).
    \]
    We claim that $\tilde{f} \circ \pi_k|_{\pi_k^{-1}(U)} = f|_{\pi_k^{-1}(U)}$, i.e.\ for arbitrary $a_j \in \C, j>k$ we claim that $\tilde{f} (z_1,\ldots,z_k) = f(z_1,\ldots,z_k,a_{k+1},a_{k+2},\ldots)$. To see this pick $v \in \C^\N$ such that $\pi_\ell (v) = 0$ for all $1\leq \ell \leq k$. Then the affine line $a(z) \coloneq (z_1,\ldots,z_k, 0,0,\ldots) + z v$ is contained in $\Omega$ and $g \colon \C \rightarrow \C, z \mapsto f(a(z))$ is a bounded holomorphic function from $\C$ to $\C$, i.e.\ constant by Liouvilles theorem. This shows that $f|_U$ is constant in the variables $a_{k+1},\ldots$ and thus we obtain the desired identity. 
  \end{proof}
  
  Having identified the germs of holomorphic functions on $\C^\N$ as a limit of the germs of holomorphic functions on finite parts of the infinite product, we expect the set of germs of holomorphic functions $\RHoloN$ to also be an inductive limit (of the germs $\RHolok$) and thus a Silva space. As a straight forward consequence of \Cref{lem:dimlim:holo} we obtain:
  
  \begin{prop}\label{prop:HSilva}
    \begin{itemize}
     \item [(a)] The space $\RHoloN$ is a Silva space as the inductive limit of the $\RHolok$.
     \item [(b)] The space $\CHoloN$ is a Silva space as the inductive limit of the $\CHolok$.
    \end{itemize}
  \end{prop}
  Indeed the above topology of the germs of meromorphic functions turns the subspace of germs of holomorphic functions into the desired inductive limit. 
  Note that the same arguments clearly carries over to the real subspace $\RHoloN$ which is a Silva space and the inductive limit of the Silva spaces $\RHolok$
  Fix for every $k \in \N$ a complex inner product (a sesquilinear form which is linear in the first component)
  \[
     Q_k\colon \C^k\times \C^k \to \C, 
  \]
  that is compatible with the canonical inclusion $\text{inc}_k \colon \C^k\hookrightarrow \C^{k+1}, k\in \N$ i.e.
  \begin{equation}\label{eq:Qiota}Q_{k+1}\circ (\text{inc}_k \times \text{inc}_k)= Q_k\quad \forall k\in \N.
  \end{equation} 
  Moreover, we shall require for every $Q_k$ that $Q_k|_{\R^k \times \R^k}$ is a real inner product.
  The map $\text{inc}_k$ induces a canonical injection $\text{inc}_k^*\colon(\C^{k+1})^*\hookrightarrow(\C^{k})^*$ by pull-back $\text{inc}_k^* (L)(v)\coloneq L(\text{inc}_k(v))$ and for any $ L_1\in(\C^{k})^*$, $ L_2\in(\C^{k+1})^*$ we have
  \begin{equation}\label{eq:iotapr} \left(\mathrm{pr}_k^* (L_1)\right)^* = \text{inc}_k (L_1^*), \quad \left(\text{inc}_k^* (L_2)\right)^* = \mathrm{pr}_k (L_2^*).
  \end{equation} 
  \begin{exa} $k=1$, $L_1= c e_1^*$, $L_1=ae_1^*+ be_2^*\in(\C^{2})^*$, $L_1^*=a e_1+ be_2$, $\mathrm{pr}_1^*(L_1)= c e_1^*\circ \mathrm{pr}_1= c e_1^*+0 e_2^*$, $\mathrm{pr}_1(L_2^*)=a e_1$, $\iota_1 (L_1^*)= ce_1+ 0 e_2$, $i_1^*(L_2)=a e_1^*$ so $\left(\mathrm{pr}_1^* (L_1)\right)^* = \iota_1 (L_1^*)$ and $(i_1^*(L_2))^*=\mathrm{pr}_1(L_2^*)$.
  \end{exa}
  
  We shall now establish the Silva space property of the space of germs, while explicitly describing the convergence in this space for later use.
  
  \begin{prop}\label{prop:MLambdainfty}
    The space $\RMeroN$ is a Silva space, thus $\varphi \colon \RMeroN \rightarrow \C$ is continuous if and only if the restrictions $\varphi_k\colon \RMerok\longrightarrow \C$ of $\varphi$ at every step is continuous. Moreover, $\varphi$ is continuous, if and only if for any sequence $(f_n)_{n\in \N}$ in $\RMerok$ for which there are linear forms $L_m\colon \C^k\to \C, m=1, \cdots, M$ such that the functions $g_n\coloneq \prod_{m=1}^M L_m(z_1,\cdots, z_k)\, f_n, n\in \N $ lie in $\RHolok$ and converge to some function $g$ in $\RHolok$, we have 
    \[
       \varphi_k(f_n)\underset{n\to \infty}{\longrightarrow} \varphi_k\left(\frac{g}{\prod_{m=1}^M L_m(z_1,\cdots, z_k)} \right). 
    \]
    
  \end{prop} 
  \begin{proof}
    Let us first note that by construction, $\RMeroN$ is the countable inductive limit of the Silva spaces $\RMerok$, see \Cref{prop:lcxv-alg}. Thus it is a Silva space due to Lemma \ref{la:ctbl_Silva}. 
    As outlined in Section \ref{rem:Silva}, if $I_k \colon \RMerok \rightarrow \RMeroN$ is the canonical inclusion, $\varphi_k \coloneq \varphi \circ I_k$ is continuous if and only if $\varphi$ is so. Moreover, as on a Silva space continuity is equivalent to sequential continuity, we pick a sequence $(f_n)_{n \in \N} \subseteq \RMeroN$ converging to some $f$. From \Cref{rem:Silva} we infer the existence of some $k $ in $ \N$ such that $(f_n)_{n \in \N}, f \in \RMerok$. Since $\RMerok$ is a Silva space, by Remark \ref{rem:pSilva} (S2) we may assume that the sequence and its limit is already contained in one of the steps $\Mer_P(B_{1/n} (0))$. Since $\Mer_P(B_{1/n} (0)) \cong \BHol (B_{1/n}(\C^k))$ as Banach spaces, the sequence $(f_n)_n$ converges if and only if $(f_n \prod_{m =1}^M L_m) \in \BHol (B_{1/n}(\C^k))$ converges in the bounded holomorphic functions for suitable linear forms $L_m$. In particular, we deduce that (sequential) continuity of $\varphi_k$ (and thus also of $\varphi$) is equivalent to the condition stated in the statement of the Proposition.
  \end{proof}

\begin{prop}
Pointwise multiplication turns $\RMeroN$ and $\RHoloN$ into locally convex algebras.
\end{prop}

\begin{proof}
 Since $\RHoloN$ is a multiplicatively closed subalgebra of $\RMeroN$, it suffices to establish continuity of the pointwise multiplication for $\RMeroN$.

 The pointwise product is defined on the Silva space $\RMeroN \times \RMeroN = \lim_{k} (\RMerok \times \RMerok)$, where we have used that the product of Silva spaces is again a Silva space as the inductive limit of the $(\RMerok \times \RMerok)$, \cite[Theorem 9]{Kom67}.
 Now restricting the pointwise multiplication to any step $(\RMerok \times \RMerok)$ of the inductive limit, it factors through the multiplication of the algebra $\RMerok$ and for this multiplication we have seen in Proposition \ref{prop:lcxv-alg} that it is continuous.
 Hence continuity of the algebra product follows at once from continuity on the steps of the limit and the property Remark \ref{rem:pSilva} (S3) of Silva spaces.
\end{proof}
 \section{Topological splitting of the space of meromorphic germs with prescribed poles and associated projections}
  We prove a topological refinement of the known algebraic decomposition (see \cite{BV} and \cite{GPZ}) of the algebra $\RMeroN $
  of meromorphic germs at zero on $\C^\N$ mapping $\R^\N$ to $\R$ with a prescribed type of linear pole given by a generating set $\newLa$ as a direct sum of the algebra $ \RHoloN$ of holomorphic germs at zero (again mapping $\R^\N$ to $\R$) and a space $\polarMeroN$ of polar germs at zero. Such a decomposition is often referred to as a minimal subtraction scheme in quantum field theory, and plays a role in toric geometry.
  \subsection{Topological splitting of $\RMeroN$}
  Let us begin with the topological splitting. For this recall from the last section that the space $\RHoloN$ of germs of holomorphic functions which map points in $\R^\N$ to $\R$ is a Silva space. 
  
  \begin{thm}\label{thm:MLambdaCN}
    {The family $Q = (Q_k)_{k\in \N}$ of inner products $Q_k$ on $\C^k$ induces a topological splitting}
    as a direct sum of locally convex vector spaces 
    \begin{equation}\label{eq:projMeroHol}\RMeroN = \RHoloN \oplus \polarMeroN.
    \end{equation}
  \end{thm}
  
  \begin{proof}
    For the steps of the inductive limit we have already seen in Theorem \ref{thm:splittingspaces} that $\RMerok = \RHolok \oplus \RMerok$.
    Note that the compatibility condition required of the families $(\newLa_k)_{k \in \N}$ and $(Q_k)_{k \in \N}$ imply that the bonding maps $\RMerok \rightarrow \Mero_{\newLa_{k+1}}(\C^{k+1})$ take $\RHolok$ and $\RMerok$ to $\Holo(\C^{k+1})$ and $\Mero^-_{\newLa_{k+1}, Q_{k+1}}(\C^{k+1})$, respectively. 
    Now inductive locally convex limits commute with finite products (see e.g.\ \cite[Theorem 9]{Kom67}). Hence we deduce that 
    \[
       \RMeroN = \lim_{\longrightarrow} \RHolok \times \lim_{\longrightarrow}  \Mero^-_{\newLa_{k}, Q_k}(\C^{k})
    \]
    Moreover, \Cref{lem:dimlim:holo} implies that $\RHoloN = \bigcup_{k \in \N} \RHolok$. Together with the splitting of the inductive limit $\RMeroN$ this shows that $\RHoloN$---as a subspace of $\RMeroN$---is isomorphic as a locally convex space to the inductive limit $ \underset{\longrightarrow}{\lim}\RHolok$. 
  \end{proof}

  \begin{rem}
  Note that the space $\polarMeroN$ is not closed under the pointwise multiplication, since for $L_1,L_2 \in \newLa$ the product $L_1/L_2 \cdot L_2/L_1$ is holomorphic, but for $L_1$ and $L_2$ mutually orthogonal the factors are contained in $\polarMeroN$.
  \end{rem}

  \subsection{A projection map onto $\RHoloN$}
  
  Recall the duality between scalar products $Q$ and $Q^\ast$, see Eq. \eqref{duality:scal}.
  \begin{la} $Q_k^*$ is compatible on $(\C^k)^*$ with the pull-back $\mathrm{pr}_k^*$ by the canonical projections
    \[
       Q_{k+1}^*\circ (\mathrm{pr}_k^* \times \mathrm{pr}_k^* )= Q_k^*\quad \forall k\in \N.
    \]
    
  \end{la}
  \begin{proof} Indeed, for any $L_1, L_2 $ in $(\C^k)^*$, by (\ref{eq:iotapr}) we have
    \begin{align*} 
      Q_{k+1}^*(\mathrm{pr}_k^* (L_1),\mathrm{pr}_k^* (L_2)) \underset{\mathrm{defn.}\,Q_{k+1}^*}{=}& Q_{k+1}((\mathrm{pr}_k^* (L_1))^*,(\mathrm{pr}_k^*(L_2))^*)\\
      \underset{(\ref{eq:iotapr})}{=}& Q_{k+1}(\text{inc}_k(L_1^*),\text{inc}_k (L_2^*))\\
      \underset{(\ref{eq:Qiota})}{=}& Q_{k}(L_1^*,L_2^*)\\
      \underset{\mathrm{defn.}\,Q_k^*}{=}& Q_{k}^*(L_1,L_2). \qedhere
    \end{align*}
  \end{proof}
  For two linear subspaces $W, W^\prime $ of $(\C^k)^* $, we write $W\,\perp^{Q_k}\,W^\prime$ if the two spaces are perpendicular with respect to the dual inner product, i.e.~$Q_k^\ast (w, w^\prime)=0$ for any $(w, w^\prime)\in W\times W^\prime$. It follows from the above discussion that 
  \begin{equation}\label{eq:perpQk}W\,\perp^{Q_k}\,W^\prime\Rightarrow \mathrm{pr}_k^*(W)\,\perp^{Q_{k+1}}\,\mathrm{pr}_k^*(W^\prime). 
  \end{equation}
  We equip the vector space $ \RMerok $ of meromorphic germs in $k$ variables with linear poles  with the symmetric binary relation $\perp^Q$ 
  \begin{equation}\label{eq:orthQk}f_1\perp^{Q_k} f_2\Leftrightarrow   \Dep(f_1)\perp ^{Q_k} \Dep(f_2).
  \end{equation} 
  It follows from (\ref{eq:perpQk}) that for any $f, f^\prime \in \RMerok$, we have
  \[
     f_1\perp^{Q_k} f_2\Rightarrow \mathrm{pr}_k^* f_1 \perp^{Q_{k+1}}\mathrm{pr}_k^* f_2.
  \]
  
  This shows that the symmetric bilinear relations $\perp^{Q_k}$ on $\RMerok$ with $ k\in \N$, induce a symmetric bilinear relation on   the inductive limit $\RMeroN$ which we denote by $\perp^Q$ so that 
  \[
     f_1\perp^Q f_2\Leftrightarrow {\mathrm{pr}}_k^* f_1\perp^{Q_{k+1}} {\mathrm{pr}}_k^* f_2\quad \forall k\in \N. 
  \]

  \begin{prop}\label{prop:piQ} 
    
    Under the assumptions and with the notations of Theorem \ref{thm:MLambdaCN}, the projection map $\pi_Q:\RMeroN\longrightarrow\RHoloN$ induced by the splitting 
    (\ref{eq:projMeroHol}) is a continuous linear map with the following partial multiplicative property:
    \begin{equation}\label{eq:partialmult}f_1\perp^Qf_2\Longrightarrow \pi_Q(f_1\, f_2)= \pi_Q(f_1)\, \pi_Q(f_2).
    \end{equation}
  \end{prop}
  \begin{proof} The algebraic splitting of the space of meromorphic functions with linear poles at zero into an algebraic direct sum of the space of holomorphic ones
    and the space spanned by the polar functions was shown (for germs) in \cite[Corollary 4.15]{GPZ} as a consequence of the existence and uniqueness of Laurent expansions. The fact that the splitting is topological is new to our knowledge and a direct consequence of Theorem \ref{thm:MLambdaCN}.
  \end{proof}
  The projection onto the holomorphic part along the polar part can be interpreted as a ``multivariate subtraction scheme'', which
  generalises the minimal subtraction projection operator for meromorphic germs in one variable.

 \section {Topological Minimal Subtraction scheme} 
  The minimal subtraction scheme corresponds to a linear map 
  $$\ev_{z=0}^{\reg}\colon \Mero(\C)\to \C$$ 
  defined by $\ev_{z=0}^{\reg} \coloneq \ev_0 \circ \pi$, where $ \pi$ is the projection onto the holomorphic part of the Laurent expansion at $0$. We call this map the {minimal evaluator} at zero.
  
  Following \cite{GPZ}, in order to generalise such a minimal evaluation to $\RMeroN$ for some  generating set $\newLa$, we use an inner product $Q=\{Q_k, k\in \N\}$ on $\C^\N$. For $k\in \N$, with the notations of (\ref{eq:orthQk}), let 
  \[
     \perp^{Q_k}\coloneq \{(f_1, f_2)\in \RMerok\times \RMerok, \quad f_1\perp^{Q_k}f_2\}.
  \]
  
  \begin{prop}
    For any generating set $\newLa=\cup_k \newLa_k$, and any $k\in\N$, the set 
    \[
       \perp_{\newLa}^{Q_k}\coloneq \perp^{Q_k}\cap\left( \RMerok\times \RMerok\right), \quad k\in \N, 
    \]
    is a closed subset of $ \RMerok\times \RMerok$. Consequently, 
    \[
       \perp_{\newLa}^{Q}\coloneq 	\bigcup_{k\in \N}\perp_{\newLa}^{Q_k} 
    \]
    is a closed subset of the Silva space $ \RMeroN\times \RMeroN$.	
  \end{prop}
  \begin{proof} 
    Recall that the space $\RMerok\times\RMerok$ is a Silva space and therefore sequential, Remark \ref{rem:pSilva} (S1), so that we can show closedness of the subset $\perp_{\newLa}^{Q_k}$ using sequences. For this purpose, we introduce two sequences
    $(f_m)_{m\in\N}$ and $(\widetilde f_m)_{m\in\N}$   in $\RMerok$ such that $(f_m)_m$ converges to $f$,  $(\widetilde f_m)_m$ converges to $\widetilde f$ and  $f_m \perp^{Q_k} \widetilde f_m$ for all $m$ in $\N$.
    It remains to show that $f\perp^{Q_k} \widetilde f$.
    Let $L$ in $ \Dep(f)$ and $\widetilde L $ in $ \Dep(\widetilde f)$ be given. We will show that $Q_k^*(L,\widetilde L)=0$.
    
    We apply   Lemma \ref{la:dep_and_silva_convergence} first to $(f_m)_m$ and then to the sequence $(\widetilde f_m)_m$ and---after choosing   two subsequences---we find  sequences $(L_m)_m$  (resp.  $(\widetilde L_m)_m$) in $(\C^k)^*$   converging  to $L$ (resp.   $\widetilde L$). Furthermore we may assume that for all  $m$ in $\N$ we have $L_m\in\Dep(f_m)$ and $\widetilde L_m\in\Dep(\widetilde f_m)$.
    
    For all $m$ in $\N$ we have $Q_k^*(L_m,\widetilde L_m)=0$.
    The bilinear map $Q_k^*$ is defined on a finite-dimensional space and therefore continuous. Hence, $Q_k^*(L,\widetilde L_m)=0$ follows.
  \end{proof}
  
  \begin{defn}\label{defn:evaluator} Let $\newLa=\cup_k \newLa_k$ be a  generating set. We call (resp. continuous) {generalised $Q$-evaluator} on $ \RMeroN $  a family of linear maps 
    \[
       {\mathcal E}_k\colon  \CMerok \longrightarrow  \C,\quad k\in \N
    \]
    with the following properties. For any $k\in \N$, 
    \begin{enumerate}
     \item ${\mathcal E}_k$ is (partially) multiplicative. For $f_1, f_2$ in  $ \RMeroN $
      \begin{equation}\label{eq:multFQ}f_1\perp^{Q_k}f_2\Longrightarrow {\mathcal E}_k(f_1\cdot f_2)= {\mathcal E}_k(f_1)\cdot {\mathcal E}_k(f_2), 
      \end{equation}
     \item the restriction of ${\mathcal E}_k$ to holomorphic germs $\RHolok$ at zero coincides with the evaluation $\ev_0: f\longmapsto f(0)$ at zero,
     \item ${\mathcal E}_k$ is compatible with the filtration ${\mathcal E}_{k+1} \vert_{\CMerok}= {\mathcal E}_k$. 
     \item (resp. ${\mathcal E}_k$ is continuous).
    \end{enumerate}
  \end{defn}
  
  An inner product $Q$ gives rise to a continuous projection $\pi_Q\colon \RMeroN\rightarrow\RHoloN$ by Proposition \ref{prop:piQ},  which  combined with the (partial) morphism property of $\pi_Q$ on meromorphic germs \eqref{eq:partialmult} and the morphism property of the evaluation at zero $\ev_0$ on holomorphic germs, leads to the following statement.
  \begin{thm}\label{thm:ev}
    For any generating set $\newLa$, the map
    \[
       {\mathcal E}_Q^{\mathrm{MS}}\coloneq \ev_0\circ \pi_Q:\RMeroN\to \C
    \]
    defines a continuous generalised evaluator, which we call \textbf{minimal generalised $Q$-evaluator}.
  \end{thm}
  With the notations of Example \ref{ex:Feynman}, let $\Mero_{\newLa^F}(\C^\N)$ denote the corresponding algebra of meromorphic germs with poles in $ \mathcal{S}_{\newLa}^F$. Such meromorphic germs naturally arise from computing Feynman diagrams, hence the choice of subscript $F$.
  \begin{exa}\label{ex:evaluatorQ}The map 
    \[
       \mathcal{E}_Q^{\text{MS}}\coloneq \ev_0\circ \pi_Q:\Mero_{\newLa^F}(\C^\N)\to \C
    \]
    defines a minimal generalised $Q$-evaluator on   $\Mero_{\newLa^F}(\C^\N)$, the algebra of meromorphic germs arising from Feynman diagrams. 
  \end{exa}
  
  Our terminology is borrowed from Speer's classical work on analytic renormalisation \cite{Speer}.
  
  For a function $f$ in $ \Mero_{\newLa^F}(\C^\N)$,  let $\Supp(f)$ denotes the subset of  $\N$ which indexes the variables on which $f$ actually depends.
  For $k $ in $\N$, we define the  sets 
  \[
     \top_k\coloneq \{(f, f')\in \RMerok\times \RMerok, \quad\Supp(f)\cap \Supp(f')=\emptyset \},
  \]
  
  \[
     \top_{\newLa^F, k} = \top_{k}\cap \left(\Mero_{\newLa^F}(\C^k)\times \Mero_{\newLa^F}(\C^k)\right),
  \]
  
  and 
  \[
     \top_{\newLa^F}\coloneq \bigcup_{k\in \N} \top_{\newLa^F, k} .
  \]
  
  In the Example in \S 3 of his paper, Speer defines an explicit family of what he calls generalised evaluators. 
  To $\sigma $ in $ \Sigma_k =\{\sigma: [[1, k]]\to [[1, k]], \ \text{ bijective}\}$, he assigns a linear map ${\mathcal E}^\sigma:\RMerok\to \C$ by
  \[
     {\mathcal E}^\sigma_k(f)\coloneq e_{\sigma(1)}\circ \cdots \circ e_{\sigma(k)}(f)= \ev_{z_{\sigma(1)}=0}^{\reg} \left(\cdots \ev_{z_{\sigma(k)}=0}^{\reg}(f(z_1, \cdots, z_k)\right) 
  \]
  built from iterated regularised evaluators $e_j\coloneq \ev_{z_{j}=0}^{\reg}$ which take the finite part at zero of the germ $f_j\in \calm(\C)$ obtained from fixing all the variables but the $j$-th one. Speer shows that the  map 
  \[
     {\mathcal E}_k^F\coloneq \frac{1}{k!}\, \sum_{\sigma \in \Sigma_k}\mathcal E_k^\sigma, \quad k\in \N, 
  \]
  defines a bilinear map 
  on $\top_{\newLa^F} $ with the following properties
  \begin{enumerate}
   \item ${\mathcal E}_k^F$ is (partially) multiplicative in the following sense
    \begin{equation}\label{eq:multF}f_1 \,\top_k\, f_2\Longrightarrow {\mathcal E}_k^F(f_1\cdot f_2)= {\mathcal E}_k^F(f_1)\cdot {\mathcal E}_k^F(f_2), 
    \end{equation}
   \item the restriction of ${\mathcal E}_k^F$ to holomorphic germs $\RHolok$ at zero coincides with the evaluation $\ev_0\colon f\longmapsto f(0)$ at zero,
   \item ${\mathcal E}_k^F$ is compatible with the filtration ${\mathcal E}_{k+1}^F\vert_{\Mero_{\newLa^F}(\C^k)}= {\mathcal E}_k^F$.
   \item  ${\mathcal E}_k^F$ is continuous. Explicitly (see the above lemma), for a sequence $(f_n)_{n\in \N}$ in $\Mero_{\newLa^F}(\C^k)$ such that there are linear forms $L_m: \C^k\to \C, m=1, \cdots, M$ for which the  functions $g_n\coloneq \prod_{m=1}^M L_m(z_1,\cdots, z_k)\, f_n, n\in \N $ lie in $\RHolok$ and converge to some function $g$ in $\RHolok$, we have 
    \[
       {\mathcal E}_k^F(f_n)\underset{n\to \infty}{\longrightarrow} {\mathcal E}_k^F\left(\frac{g}{\prod_{m=1}^M L_m(z_1,\cdots, z_k)} \right). 
    \]
    
  \end{enumerate}
  \begin{rem}For the canonical inner product $Q$ on $\C^\N$ we have 
    \[
       f_1\top_k f_2\Longrightarrow f_1\perp^{Q_k}f_2
    \]
    so that Speer's multiplicativity condition (\ref{eq:multF}) is more stringent than (\ref{eq:multFQ}).  Thus, the generalised evaluator ${\mathcal E}^F$ does not define a generalised $Q$-evaluator on $\Mero_{\newLa^F}(\C^\N)$. 
  \end{rem}
  
  \appendix 
  
 \section{Proof details for Section \ref{sect:meroCk}}\label{App:proofdetails}
  In this appendix we provide the proof for Lemma \ref{la:dep_and_silva_convergence} and Lemma \ref{lem:splittingoff} . 
  
  \subsection{Proof of Lemma \ref{la:dep_and_silva_convergence}}
  Recall the statement of the Lemma: \emph{Fix $k\in\N$ and assume that $f_1,f_2,\ldots $ is a sequence of elements in $\RMerok$ which converges to $f$ in $\RMerok$. Then for any $\ell$ in $\Dep(f)$, there is a subsequence $(f_{m_j})_{j\in \N}$ and a sequence $\ell_{m_j}$ in $ \Dep(f_{m_j}), j\in \N$ which converges to $\ell$.}
  
  \begin{proof}[Proof of Lemma \ref{la:dep_and_silva_convergence}]
    For $f$ and $\ell$ and a sequence as in the statement of the Lemma let us build a subsequence of $(f_m)_{m\in \mathbb N}$ together with the sequence $\ell_m\in \Dep(f_m)$.\medskip
    
    \textbf{Step 1: The independence subspaces} $\Indep(f_m)$ and $\Indep(f)$.\\ 
    By Corollary \ref{cor:silva_and_uniform_convergence} there is an open convex set $U\subseteq\C^k\setminus\{0\}$ such that each $f_m$ is bounded and holomorphic on $U$ and the sequence $f_m|_U$ converges uniformly to $f|_U\in\BHol(U)$. With a slight abuse of notation we denote by $f_m$ and $f$ again $f_m|_U$ and $f|_U$, respectively.
    Note that the (in)dependence subspaces of $f_m$ and $f$ do not depend on $U$.\medskip
    
    \textbf{Step 2: An orthonormal basis for $\Indep(f)$.}\\
    For each $m\in\N$ the dimension $d_m\coloneq \dim \Indep(f_m)$ of the space $\Indep(f_m)$ is a number in the finite set $\{0,\ldots,k\}$. Hence, after chosing a subsequence, we may assume without loss of generality that $d\coloneq \dim \Indep(f_m)$ is constant.
    For each $m\in\N$ we may chose a $Q$-orthonormal basis
    \[
       B_m\coloneq \left( b_1^{(m)},\ldots, 
       b_k^{(m)}    \right)
    \]
    of $\C^k$ such that the first $d$ vectors form a $Q$-orthonormal basis of $\Indep(f_m)$. The sequence $(B_m)_m$ can be regarded as a bounded sequence in the finite-dimensional space $(\C^k)^k$ and therefore has a converging subsequence which we will again denote by $(B_m)_m$ to keep the notation relatively simple.
    
    The defining properties of a $Q$-orthonormal basis are stable under limits (due to the continuity of $Q$), whence
    \[
       B=\left(b_1,\ldots, b_k\right)
       \coloneq \lim_{m\to\infty} B_m
       = \left( \lim_{m\to\infty} b_1^{(m)},\ldots, \lim_{m\to\infty} b_k^{(m)}    \right)
    \]
    is also a $Q$-orthonormal basis of $\C^k$. By the definition of the independent subspaces (see Eq. (\ref{eq:indspace})), for every $m\in\N$ and $j\leq d$ we have $D_{b_j^{(m)}}f_m=0$.
    Since differentiation is continuous with respect to the topology of uniform convergence on holomorphic functions, we infer $D_{b_j}f=0$ and $b_j\in\Indep(f)$ for all $j\leq d$.
    This shows that $\mathrm{span}(b_1,\ldots,b_d)\subseteq\Indep(f)$, from which it follows that $(b_1,\ldots,b_d)$ is an orthonormal basis for $\Indep(f)$.\medskip
    
    \textbf{Step 3: Approximating $\ell$ in $\Dep(f)$.}\\ 
    By Riesz' theorem, there is a vector $v\in\C^k$ such that $\ell =Q(\cdot,v)$. Since $\Dep(f)$ is the annulator of $\Indep(f)$, we obtain:
    \[
       v\in \Indep(f)^\perp \subseteq \mathrm{span}(b_1,\ldots,b_d)^\perp = \mathrm{span}(b_{d+1},\ldots, b_k).
    \]
    Thus there exist scalars $\alpha_{d+1},\ldots,\alpha_k\in\C$ such that
    $
    v = \sum_{j=d+1}^k \alpha_j b_j.
    $
    For each $m\in\N$ we define
    $
    v_m\coloneq \sum_{j=d+1}^k \alpha_j b_j^{(m)}\in\Indep(f_m)^\perp
    $
    and obtain a sequence converging to $v$. Since each $v_m$ is $Q$-orthogonal to $\Indep(f_m)$, it corresponds to an element $\ell_m$ in $\Dep(f_m)$ using Riesz' representation theorem. We have therefore built a sequence $(\ell_m)_{m\in \N}$ in the dual space $(\C^k)^*$ which converges to $\ell$ as desired.
  \end{proof}
  
  \subsection{Proof of Lemma \ref{lem:splittingoff}}
  
  Recall the statement of the Lemma: \emph{Given $f\in \BHol^\C (B_{1/n}(\C^k)) $, $L \in \newLa$ and $n \in \N$ there is an integer $m > n$ such that the maps $h:= f\circ \pr_L|_{B_{1/m}(\C^k)} $ and $g=\left.\left(f- f\circ\pr_L\right) / L\right|_{B_{1/m}(\C^k)}$ are bounded holomorphic functions on $B_{1/m}(\C^k)$ and the associated mapping  
  \[
     \theta_{n,m,Q}^L \colon \BHol^\C (B_{1/n}(\C^k)) \rightarrow \BHol^\C (B_{1/m}(\C^k))^2,\quad f= L\cdot g + h \mapsto (g,h),
  \]
  is continuous linear.}
  
  \begin{proof}[Proof of Lemma \ref{lem:splittingoff}]	
    To distinguish supremum norms we write
    \[
       \lVert f \rVert_{R, \infty} \coloneq \sup_{z\in B_{R}(\C^k)} |f(z)| \text{ for } R>0.
    \]
    Pick now $f $ in $ \BHol^\C (B_{1/n}(\C^k))$ and define $h = f \circ \pr_L|_{B_{1/n}(\C^k)}$. Clearly $h$ is holomorphic and bounded by 
    \begin{align}\label{hestimate}
      \lVert h\rVert_{1/n,\infty} \leq \lVert f \rVert_{1/n,\infty}.
    \end{align}
    By definition $f-f\circ \pr_L$ takes values in the closed linear subspace 
    \[
       V_L \coloneq \{g \in \BHol^\C (B_{1/n}(\C^k))\mid g|_{\ker L} \equiv 0 \}.
    \]
    Now the subtraction map $s_L \colon \BHol^\C(B_{1/n}(\C^k)) \rightarrow V_L, s_L(f) = f - f\circ \pr_L \in V_L$, is continuous linear since $\lVert s_L(f)\rVert_{1/n,\infty} \leq 2\lVert f\rVert_{1/n,\infty}$ by \eqref{hestimate}. \medskip{}
    
    \textbf{Step 1:} For $\kappa\in V_L$, $z \mapsto \kappa(z)/L(z)$ is holomorphic on $B_{1/(n+1)}(\C^k)$.\\
    The map $L$ is a linear form on $\C^k$, whence the Riesz representation theorem yields $v \in \C^k$ with $L(x) = Q(x,v)$. 
    Since $L$ is continuous we can pick $0 < s < 1/n$ such that for $z \in B_{s}(\C^k)$ the $Q$-orthogonal decomposition $z = \pr_L (z) + (L(z)/\lVert v\rVert^2)v$ satisfies $\lVert \pr_L(z)\rVert + |L(z)|/\lVert v\rVert < 1/n$. Hence we may apply for every $z$ in $ B_s(\C^k)$ and $\kappa $ in $ V_L$ the mean value theorem:
    \begin{align*}
      \kappa(z) & = \kappa(\pr_L (z) + (L(z)/\lVert v\rVert^2)v) -\kappa(\pr_L (z)) \\ &= \int_0^1 d\kappa (\pr_L (z)+\lambda (L(z)/\lVert v\rVert^2)v;(L(z)/\lVert v\rVert^2)v) \mathrm{d}\lambda \\ &= L(z) \int_0^1 d\kappa (\pr_L (z)+\lambda (L(z)/\lVert v\rVert^2)v;v/\lVert v\rVert^2)\mathrm{d}\lambda
    \end{align*}
    hence $\kappa(z)/L(z)$ makes sense as a holomorphic mapping on $B_{s}(\C^k)$ and is clearly bounded on $B_{R}(\C^k)$ for every $R<s$. Thus it makes sense to define for $R <s$ the map 
    \begin{equation}\label{eq:deltaR}\delta_R \colon V_L \rightarrow \BHol^\C (B_{R}(\C^k)),  \kappa \mapsto \kappa /L|_{B_{R}(\C^k)}.
    \end{equation}
    
    \textbf{Step 2:} The map $\delta_R(\kappa)\coloneq \kappa / L|_{B_{R}(\C^k)}$ is continuous for a suitable $R$.     
    let $f \in V_L$ and write $f = L\cdot g$, i.e.\ $g|_{B_{R}(\C^k)} = \delta_R (f)$. Let now $z \in \C^k$ with $\lVert z\rVert =r <s$ for the $s$ as in Step 1. As the projection has operator norm $1$, we have by construction that $\lVert\pr_L (z)\rVert \leq r < 1/n$. Setting $\varepsilon_r = \lVert v\rVert(1/n - r)$ a quick computation yields $\lVert \pr_L (z)+(\lambda/\lVert v\rVert^2) v\rVert \leq 1/n$ for $|\lambda |< \varepsilon_r$. Hence it makes sense to define the following holomorphic function of one variable
    \[
       \phi_z \colon \{\lambda \in \C \mid | \lambda| < \varepsilon_r\} \rightarrow \C , \quad \lambda \mapsto g(\pr_L (z) + (\lambda /\lVert v\rVert^2) v).
    \]
    From \eqref{Restimate} in \Cref{la:oneD} we infer that 
    $
    \sup_{|\lambda| < \varepsilon_r} |\lambda \phi_z(\lambda)| = \varepsilon_r\lVert \phi_z\rVert_{\varepsilon_r,\infty}.
    $
    Let us assume that $\lVert z\rVert \leq r < s$. Then from $\lVert z-\pr_L (z)\rVert =|L(z)|/\lVert v\rVert$ we deduce that $|L(z)| \leq 2r\lVert v\rVert$. 
    As the constant $\varepsilon_r$ is growing for smaller $r$, we can choose $m \in \N$ such that $R := 1/m < \min\{s,\varepsilon_R, \frac{\varepsilon_r}{2\lVert v\rVert}\}$. Summing up this yields the estimate
    \begin{align*}
      | g(z)| = |\phi_z(L(z))| \leq \frac{\varepsilon_r}{\varepsilon_r} \lVert \phi_z\rVert_{\varepsilon_r,\infty} = \frac{1}{\varepsilon_R} \sup_{|\lambda| \leq \varepsilon_r}| \lambda \phi_z(\lambda)|
      \leq \frac{\lVert L \cdot g\rVert_{1/n,\infty}}{1/n-s}.
    \end{align*}
    As $z$ was arbitrary with $\lVert z \rVert < 1/m$, we infer that 
    \[
       \lVert g\rVert_{1/m,\infty} \leq \frac{\lVert L \cdot g\rVert_{1/n,\infty}}{1/n-s} =  \frac{\lVert f\rVert_{1/n,\infty}}{1/n-s}.
    \]
    Hence the map $\delta_R$ defined in (\ref{eq:deltaR}) is continuous for every $R \leq 1/m$.\medskip
    
    \textbf{Step 3:} $\theta_{n,Q}^L$ is continuous linear. 
    First recall that the restriction map $r^n_m \colon \BHol^\C (B_{1/n}(\C^k)) \rightarrow \BHol^\C(B_{1/m}(\C^k))$ is continuous linear. We can then write $\theta_{n,Q}^L = (\delta_{1/m} \circ s_L, r^n_m -L|_{B_{1/m}(\C^k)} \cdot \delta_{1/m} \circ s_L)$. This mapping makes sense and the first component is continuous by Step 1 and 2. Exploiting that $\BHol^\C(B_{1/m}(\C^k))$ is a Banach algebra with multiplication given by the pointwise multiplication of functions, Step 1-2 show that $\theta^L_{n,Q}$ is indeed continuous.
  \end{proof}

    \textbf{Acknowledgement}
  We are grateful to the late Berit Stens\o{}nes for insightful comments on holomorphic functions of several variables. Furthermore, we thank Li Guo for his very useful comments on a preliminary version of the paper. We thank the anonymous referee for insightful comments on a preliminary version of this work.
  A.S. thanks Nord university in Levanger, where he was employed while part of the present work was conducted.
  \medskip
  
  \textbf{Authors' information}\\
  R. Dahmen, Karlsruhe Institute of Technology, Germany, \href{mailto:rafael.dahmen@kit.edu}{rafael.dahmen@kit.edu}  \\[.25em]
  \noindent
  S. Paycha, Potsdam University, Germany, \href{mailto:paycha@math.uni-potsdam.de}{paycha@math.uni-potsdam.de} \\[.25em]
  \noindent
  A. Schmeding, NTNU Trondheim, Norway, \href{mailto:alexander.schmeding@ntnu.no}{alexander.schmeding@ntnu.no} 
  \bibliography{meromult}
  %

  
\end{document}